\documentclass[10pt,twocolumn,twoside]{IEEEtran}
\usepackage[utf8]{inputenc}
\usepackage{amsmath,amsthm,amsfonts,amssymb,bm,accents,mathtools}
\usepackage{graphicx,cite,enumerate,afterpage,url}
\graphicspath{{Images/}}

\usepackage{color}
\definecolor{pennblue}{cmyk}{1,0.65,0,0.30}
\definecolor{pennred}{cmyk}{0,1,0.65,0.34}
\definecolor{mygreen}{rgb}{0.10,0.50,0.10}

\newcommand \blue[1]        {#1}

\usepackage{blindtext}

\newcommand{\includesvg}[2][scale=1]{\includegraphics[#1]{#2.pdf}}

\usepackage[english]{babel}

\usepackage{algpseudocode,algorithm}
\algrenewcommand\algorithmicdo{}
\algrenewtext{EndFor}{\algorithmicend}
\algrenewtext{EndProcedure}{\algorithmicend}

\newtheorem{theorem}{Theorem}
\newtheorem*{theorem*}{Theorem}
\newtheorem{proposition}{Proposition}
\newtheorem{corollary}{Corollary}
\newtheorem{lemma}{Lemma}
\theoremstyle{definition}
\newtheorem{definition}{Definition}

\newtheorem*{problem*}{Problem}
\newtheorem{remark}{Remark}
\newtheorem*{remark*}{Remark}
\newtheoremstyle{assume}
  {3pt}% measure of space to leave above the theorem. e.g.: 3pt
  {3pt}% measure of space to leave below the theorem. e.g.: 3pt
  {}% name of font to use in the body of the theorem
  {}% measure of space to indent
  {\bf}% name of head font
  {}% punctuation between head and body
  { }% space after theorem head; " " = normal interword space
  {\thmname{#1}.\thmnumber{#2}\thmnote{ \textnormal{(\textit{#3})}}}% Manually specify head
\theoremstyle{assume}

\DeclareMathOperator{\E}{\mathbb{E}}

\DeclareMathOperator{\trace}{Tr}
\DeclareMathOperator{\diag}{diag}

\DeclareMathOperator*{\argmin}{argmin}

\DeclareMathOperator*{\minimize}{minimize}

\DeclareMathOperator{\subjectto}{subject\ to}

\DeclareFontFamily{U}{mathb}{\hyphenchar\font45}
\DeclareFontShape{U}{mathb}{m}{n}{
<-6> mathb5 <6-7> mathb6 <7-8> mathb7
<8-9> mathb8 <9-10> mathb9
<10-12> mathb10 <12-> mathb12
}{}
\DeclareSymbolFont{mathb}{U}{mathb}{m}{n}
\DeclareMathSymbol{\llcurly}{\mathrel}{mathb}{"CE}
\DeclareMathSymbol{\ggcurly}{\mathrel}{mathb}{"CF}

\newcommand{\norm}[1]{\ensuremath{\left\| #1 \right\|}}
\newcommand{\abs}[1]{\ensuremath{{\left\vert #1 \right\vert}}}

\newcommand{\floor}[1]{\left \lfloor #1 \right \rfloor}

\def\nnil{\nil}
\newcounter{prob}
\newenvironment{prob}[1][\nil]{%
	\def\tmp{#1}
	\equation
	\ifx\tmp\nnil
		\refstepcounter{prob}
		\tag{P\Roman{prob}}
	\else
		\tag{\tmp}
	\fi
	\aligned%
}{%
	\endaligned\endequation%
}

\newcommand{\calA}{\ensuremath{\mathcal{A}}}
\newcommand{\calB}{\ensuremath{\mathcal{B}}}
\newcommand{\calC}{\ensuremath{\mathcal{C}}}

\newcommand{\calE}{\ensuremath{\mathcal{E}}}

\newcommand{\calG}{\ensuremath{\mathcal{G}}}

\newcommand{\calI}{\ensuremath{\mathcal{I}}}

\newcommand{\calL}{\ensuremath{\mathcal{L}}}

\newcommand{\calS}{\ensuremath{\mathcal{S}}}
\newcommand{\calT}{\ensuremath{\mathcal{T}}}
\newcommand{\calU}{\ensuremath{\mathcal{U}}}
\newcommand{\calV}{\ensuremath{\mathcal{V}}}
\newcommand{\calW}{\ensuremath{\mathcal{W}}}
\newcommand{\calX}{\ensuremath{\mathcal{X}}}
\newcommand{\calY}{\ensuremath{\mathcal{Y}}}
\newcommand{\calZ}{\ensuremath{\mathcal{Z}}}

\newcommand{\bzero}{\ensuremath{\bm{0}}}
\newcommand{\bA}{\ensuremath{\bm{A}}}
\newcommand{\bB}{\ensuremath{\bm{B}}}

\newcommand{\bD}{\ensuremath{\bm{D}}}

\newcommand{\bG}{\ensuremath{\bm{G}}}
\newcommand{\bH}{\ensuremath{\bm{H}}}
\newcommand{\bI}{\ensuremath{\bm{I}}}

\newcommand{\bL}{\ensuremath{\bm{L}}}
\newcommand{\bM}{\ensuremath{\bm{M}}}

\newcommand{\bP}{\ensuremath{\bm{P}}}
\newcommand{\bQ}{\ensuremath{\bm{Q}}}
\newcommand{\bR}{\ensuremath{\bm{R}}}

\newcommand{\bW}{\ensuremath{\bm{W}}}

\newcommand{\bb}{\ensuremath{\bm{b}}}

\newcommand{\bu}{\ensuremath{\bm{u}}}

\newcommand{\bw}{\ensuremath{\bm{w}}}
\newcommand{\bx}{\ensuremath{\bm{x}}}
\newcommand{\by}{\ensuremath{\bm{y}}}
\newcommand{\bz}{\ensuremath{\bm{z}}}
\newcommand{\bPi}{\ensuremath{\bm{\Pi}}}

\newcommand{\bSigma}{\ensuremath{\bm{\Sigma}}}

\newcommand{\setR}{\ensuremath{\mathbb{R}}}

\def\st/{\textsuperscript{st}}
\def\nd/{\textsuperscript{nd}}
\def\rd/{\textsuperscript{rd}}
\def\th/{\textsuperscript{th}}

\newcommand{\bxb}{\ensuremath{\bm{{\bar{x}}}}}
\newcommand{\bPb}{\ensuremath{\bm{{\bar{P}}}}}
\newcommand{\bPu}{\ensuremath{\bm{{\underaccent{\bar}{P}}}}}
\newcommand{\bbt}{\ensuremath{\bm{{\tilde{b}}}}}
\newcommand{\bPt}{\ensuremath{\bm{{\tilde{P}}}}}

\newcommand{\calVbar} {\ensuremath{\overline{\calV}}}

\title{Approximately Supermodular Scheduling\\Subject to Matroid Constraints}

\author{Luiz~F.~O.~Chamon, Alexandre~Amice, and Alejandro~Ribeiro%
\thanks{Department of Electrical and Systems Engineering, University of Pennsylvania.
 e-mail: \mbox{\texttt{\{luizf, amice, aribeiro\}@seas.upenn.edu}}.
Part of the results in this paper appeared in~\cite{Chamon19m}.}%
}

\begin{document}

\maketitle

\begin{abstract}

Control scheduling refers to the problem of assigning agents or actuators to act upon a dynamical system at specific times so as to minimize a quadratic control cost, such as the objectives of the Linear-quadratic-Gaussian~(LQG) or Linear Quadratic Regulator~(LQR) problems. When budget or operational constraints are imposed on the schedule, this problem is in general NP-hard and its solution can therefore only be approximated even for moderately sized systems. The quality of this approximation depends on the structure of both the constraints and the objective. This work shows that greedy control scheduling is near-optimal when the constraints can be written as an intersection of matroids, algebraic structures that encode requirements such as limits on the number of agents deployed per time slot, total number of actuator uses, and duty cycle restrictions. To do so, it proves that the LQG cost function is $\alpha$-supermodular and provides a new $\alpha/(\alpha + P)$-optimality certificates for the greedy minimization of such functions over an intersection of $P$~matroids. These certificates are shown to approach the~$1/(1+P)$ guarantee of supermodular functions in relevant settings. These results support the use of greedy algorithms in non-supermodular quadratic control problems as opposed to typical heuristics such as convex relaxations and surrogate figures of merit, e.g., the~$\log\det$ of the controllability Gramian.

\end{abstract}

\section{INTRODUCTION} \label{S:intro}

As the number and complexity of interconnected devices grows, fully observing and/or actuating these dynamical systems becomes infeasible. Choosing when and where to sense and act then becomes a fundamental problem with significant impact on the system performance. Many of these problems can be cast as discrete optimization programs in which a subset of available elements~(sensors, tasks, or actuators) is chosen to optimize a control objective, such as quadratic estimation or control figures of merit~\cite{Gerkey04a, Shamaiah10g, Summers16o, Chamon20a}. Imposing constraints to meet cost, power, and/or communication requirements then leads to NP-hard problems whose solutions can only be approximated for anything beyond small-scale systems~\cite{Das11s, Sagnol13a, Krause08n, Ranieri14n, Tzoumas16m, Zhang17s, Olshevsky18o}.

In this work, we tackle the control scheduling problem in which agents or actuators are assigned to act at specific time instants so as to minimize a quadratic control cost. A common approach used in this context is convex relaxation, typically including a sparsity promoting regularization~\cite{Weimer08a, Joshi09s, Dhingra14a, Rusu18s}. Though practical, these methods lack approximation guarantees and may therefore result in solutions with poor performance. What is more, solving large-scale optimization problems can be challenging even when they are convex, especially when semidefinite relaxations are employed. Another avenue is to build the schedule using discrete optimization methods, such as tree pruning~\cite{Vitus12o} or greedy search~\cite{Jawaid15s, Singh17s, Jadbabaie18d}. The latter solution is particularly attractive in practice due to its low complexity and iterative nature: actuators or agents are matched to time instants one at a time by selecting the feasible match that most reduces the objective.

Another advantage of greedy methods is that they enjoy near-optimal guarantees. Two conditions are often used to derive these certificates. The first is supermodularity, a diminishing return property displayed by set functions such as the log determinant or the rank of the controllability Gramian~\cite{Summers16o, Tzoumas16m}. The second, is that the problem requirements can be mapped to a combinatorial structure known as a \emph{matroid}~\cite[Ch.~39]{Schrijver03c}. These include, for instance, constraints on the maximum number of actuators/agents selected per time instant. Under these assumptions, greedy scheduling yields~$1/2$-optimal solutions~\cite{Fisher78a}~(see Remark~\ref{R:near_optimal}). Algorithms based on continuous extensions can be used to improve this factor to~$1-1/e$~\cite{Calinescu11m}.

In practice, however, these assumptions are quite restrictive. First, quadratic control objectives used to obtain linear-quadratic-Gaussian~(LQG) or linear quadratic~(LQR) regulators are only supermodular under stringent conditions~\cite{Singh17s, Olshevsky18o}. Second, real-world applications typically have requirements too general to be mapped onto a single matroid. For instance, there are often restrictions on both the number of actuators per time instant and the total number of uses of each actuator, especially in power-starved systems. These issues are often side-stepped using supermodular surrogates such as the log determinant~\cite{Jawaid15s, Tzoumas16m}, although it is a poor substitute for quadratic costs in general~\cite[Remark~1]{Chamon20a}. Alternatively, guarantees have been obtained under laxer supermodularity conditions. They are, however, only available for the uniform matroid constraint found in selection problems~\cite{Das11s, Sviridenko14o, Chamon16n, Chamon17a, Bian17g, Summers19p, Chamon20a}.

In this paper, we set out to provide guarantees for greedy control scheduling. We start by describing the LQG/LQR scheduling problem and posing it as a constrained set function optimization program~(Section~\ref{S:problem}). We then proceed to describe condition under which this problem is tractable. We first show how greedy solutions can be obtained if the constraints are an intersection of matroids~(Section~\ref{S:matroid}) and derive near-optimal certificates for these solutions when the objective is approximately supermodular~(Sections~\ref{S:alphaSM} and~\ref{S:nearOptimal}). \blue{These results generalize those from~\cite{Chamon19m, Chamon20a} to the case of multiple, generic matroid constraints and provide stronger guarantees than those in~\cite{Chamon19m} in the case of a single matroid. However, the approximate supermodular bounds from~\cite{Chamon20a} do not hold directly for the LQG cost function.} We conclude by proving that it is nevertheless approximately supermodular, yielding explicit guarantees for the greedy control scheduling algorithm~(Section~\ref{S:alphaLQR}). We argue that these guarantees improve in scenarios of practical interest by illustrating their typical values in simulations and motivate the usefulness of this new result in an agent dispatch application~(Section~\ref{S:sims}).

\blue{\begin{remark}\label{R:near_optimal}
In the computational complexity literature, the term \emph{near-optimal} is used to refer to a solution that is within a non-zero factor of the optimal value of a mathematical program. Explicitly, let~$M^\star \geq 0$ denote the optimal value of a minimization problem~$M$. Then, $x$ is a~$\rho$-optimal solution~(or, equivalently, has an approximation factor of~$\rho$), $0 < \rho \leq 1$, if $M^\star \geq \rho M(x)$. An algorithm~$A$ that yields $\rho$-optimal solutions is also said to be $\rho$-optimal. Note that if~$\rho = 1$, then~$x$ is optimal and~$A$ \emph{solves} problem~$M$~\cite{Korte12c}. If, additionally, $A$ is polynomial-time~($A \in \text{P}$) and~$\rho$ is a constant factor, i.e., $\rho$ is independent of the problem instance, then~$M$ is said to be approximable~($M \in \text{APX}$). Observe that while this work shows that greedy control scheduling is near-optimal by bounding its approximation factor~$\rho$~(Corollary~\ref{T:guarantee}), it does not prove that the control scheduling problem is in APX~(since~$\rho$ depends on system parameters). In fact, it is likely not since its dual problem, sensor selection for Kalman filtering, does not admit a constant factor approximation unless~$\text{P} = \text{NP}$~\cite{Ye18c}.
\end{remark}}

\section{PROBLEM FORMULATION}
	\label{S:problem}

Consider a discrete-time, linear dynamical system and let~$\calV_k$ denote the \blue{finite} set of inputs available at time~$k$ as illustrated in Figure~\ref{F:sets}. Depending on the context, these abstract inputs represent actuators, agents, or both. Suppose that at each time instant, only a subset~$\calS_k \subseteq \calV_k$ of inputs is used, so that the states~$\bx_k \in \setR^n$ of the system evolve according to
\begin{equation}\label{E:dynamics}
	\bx_{k+1} = \bA_k \bx_k + \sum_{i \in \calS_k} \bb_{i,k} u_{i,k} + \bw_k
		\text{,}
\end{equation}
where, for each time~$k$, $\bA_k$ denotes the state transition matrix, $\bb_{i,k} \in \setR^n$ is a vector representing the effect of applying the control action~$u_{i,k}$ to the~$i$-th input, and~$\bw_k$ is a zero-mean Gaussian vector that models the process noise. We assume that~$\{\bw_j,\bw_k\}$ are independent for~$j \neq k$ and that their covariance matrices~$\E \bw_k \bw_k^T = \bW_k$ are either positive definite~($\bW_k \succ 0$) for all~$k$ or zero~($\bW_k = \bzero$) for all~$k$, which accounts for the deterministic dynamics case.

\blue{A control schedule is obtained by combining the sequence of active inputs~$\calS_k \subseteq \calV_k$ at each instant~$k$. In order to formally account for complex constraints, we will assume without loss of generality that~$\calV_j \cap \calV_k = \emptyset$ for all~$j \neq k$. Any input~$v$ available at more than a single time instant can be represented by unique copies, e.g., as~$v^j \in \calV_j$ and~$v^k \in \calV_k$. We then let~$\calVbar = \calV_0 \cup \calV_1 \cup \dots \cup \calV_{N-1}$ be the set of all actuators available over the~$N$-steps time window~$[0, N-1]$ and define a \emph{schedule} as a single set~$\calS \in \calVbar$. The set of active inputs at instant~$k$ is then recovered as~$\calS_k = \calS \cap \calV_k$. These definitions are illustrated in Figure~\ref{F:sets}.}

Given a schedule~$\calS \subseteq \calVbar$, designing the control actions~$u_{i,k}$ reduces to a classical optimal control problem, since~\eqref{E:dynamics} describes a well-defined, time-varying dynamical system. In particular, we consider the linear-quadratic-Gaussian~(LQG) control problem
\begin{multline} \label{E:LQG}
	V^\star(\calS) = \min_{\calU(\calS)}\ \E \Bigg[
		\sum_{k=0}^{N-1} \left( \bx_k^T \bQ_k \bx_k
			+ \sum_{i \in \calS \cap \calV_k} r_{i,k} u_{i,k}^2 \right)
	\\
	{}+ \bx_N^T \bQ_N \bx_N \Bigg]
		\text{,}
\end{multline}
where~$\calU(\calS) = \{u_{i,k} \mid i \in \calS \cap \calV_k,\ k=0,\dots,N-1\}$ is the set of valid control actions, $\bQ_k \succ 0$ for all~$k$ are the state weights, and~$r_{i,k} > 0$ for all~$i$ and~$k$ are the input weights. The relative values of these weights describe the trade-off between state regulation~($\bQ_k$) and input cost~($r_{i,k}$). The expectation in~\eqref{E:LQG} is taken with respect to the process noise sequence~$\{\bw_k\}$ and the initial state~$\bx_0$, assumed to be a Gaussian random variable with mean~$\bxb_0$ and covariance~$\bSigma_0 \succ 0$. When~$\bw_k = \bzero$ for all~$k$, \eqref{E:LQG} reduces to the linear quadratic regulator~(LQR) problem. It is useful to recall that~\eqref{E:LQG} has a closed-form solution that we describe in the following proposition.

\begin{proposition}\label{T:LQG}
	Given a schedule~$\calS \subseteq \calVbar$, the optimal value~$V^\star(\calS)$ of the LQG problem in~\eqref{E:LQG} can be written as
	\begin{equation}\label{E:optimalCost}
		V^\star(\calS) = \trace\left[ \bSigma_0 \bP_0(\calS) \right]
			+ \sum_{k = 0}^{N-1} \trace\left[ \bW_{k} \bP_{k+1}(\calS) \right]
			\text{,}
	\end{equation}
	where the~$\bP_k(\calS)$ are obtained via the backward recursion
	\begin{equation}\label{E:P}
		\bP_k(\calS) = \bQ_k
			+ \bA_{k}^T \left( \bP_{k+1}^{-1}(\calS)
			+ \sum_{i \in \calS \cap \calV_k} r_{i,k}^{-1} \bb_{i,k}^{} \bb_{i,k}^T \right)^{-1} \bA_{k}
			\text{,}
	\end{equation}
	starting with~$\bP_N = \bQ_N$.
\end{proposition}

\begin{proof}
This result follows directly from the classical dynamic programming argument for the LQG by taking into account only the active inputs in~$\calS$~(e.g., see~\cite{Bertsekas17d}). For ease of reference, we provide a derivation in the appendix.
\end{proof}

\begin{figure}[tb]
	\centering
	\includesvg[width=\columnwidth]{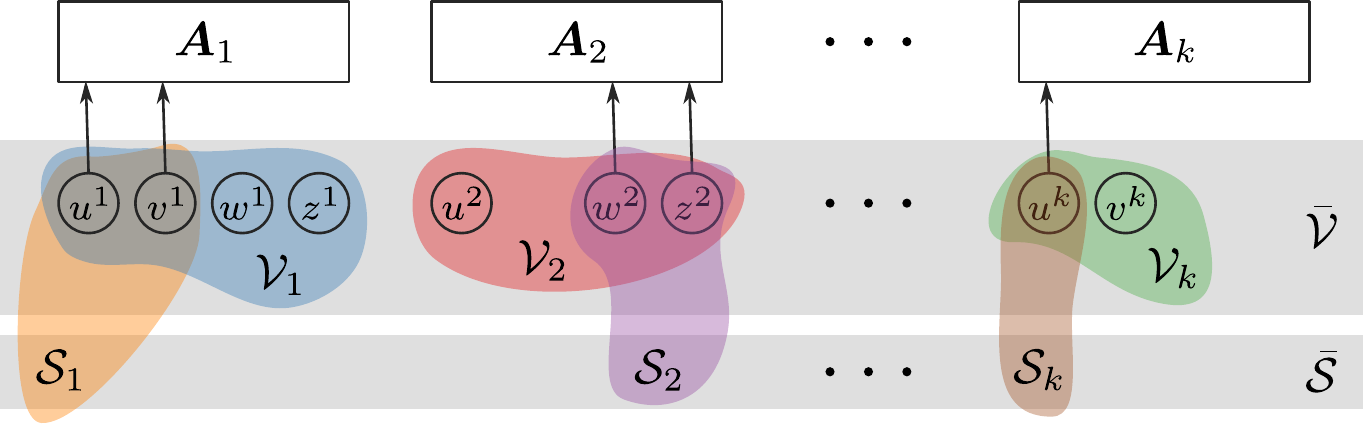}
	\caption{Illustration of time-specific and overall input sets and schedules.}
	\label{F:sets}
\end{figure}

\emph{Control scheduling} refers to the problem of finding a schedule~$\calS$ that minimizes the control cost in~\eqref{E:LQG} subject to time-input constraints, i.e.,
\begin{prob}\label{P:actuatorScheduling}
	\minimize_{\calS \subseteq \calVbar}&
		&&J(\calS) \triangleq V^\star(\calS) - V^\star(\emptyset)
	\\
	\subjectto& &&\calS \in \calI_p
		\text{,} \quad p = 1,\dots,P
		\text{,}
\end{prob}
where~$\calI_p \subseteq 2^{\calVbar}$ are families of subsets of~$\calVbar$ that enumerate admissible schedules and~$2^{\calX}$ denotes the power set of~$\calX$~(the collection of all its finite subsets). Typical scheduling requirements include (i)~limits on the total number of control actions, (ii)~limits on the number of inputs used per time instant, (iii)~restrictions on the consecutive use of inputs, and combinations thereof. Observe that the constant~$V^\star(\emptyset)$ in the objective of~\eqref{P:actuatorScheduling} does not affect the solution of the optimization problem. It is used so that~$J(\emptyset) = 0$, which simplifies the presentation of our approximation certificates.

Problem~\eqref{P:actuatorScheduling}, however, is NP-hard even for constraints as simple as a budget on the total number of control actions~($p = 1$ and~$\calI_1 = \{\calS \subseteq \calVbar \mid \abs{\calS} \leq s\}$)~\cite{Krause08n, Das11s, Sagnol13a, Ranieri14n, Tzoumas16m, Zhang17s, Olshevsky18o}. In the following section, we therefore examine an abstract version of~\eqref{P:actuatorScheduling} to determine when it is possible to efficiently obtain approximate solutions for it, focusing on greedy methods. We will show that when the~$\calI_p$ are independent sets of matroids~(Section~\ref{S:matroid})---algebraic structures that generalize the notions of linear independence in vector spaces---and the cost function is $\alpha$-supermodular~(Section~\ref{S:alphaSM})---a relaxed diminishing return property of set functions---, greedy methods yield~$[\alpha/(\alpha + P)]$-optimal schedules~(Theorem~\ref{T:approx}). We then proceed to bound~$\alpha$ for the cost~$J$~(Section~\ref{S:alphaLQR}) and give explicit guarantees for greedy solutions of~\eqref{P:actuatorScheduling} in terms of known parameters of the problem~(Corollary~\ref{T:guarantee}).

\section{WHEN IS CONTROL SCHEDULING (APPROXIMATELY) TRACTABLE?}
\label{S:tractability}

The complexity of~\eqref{P:actuatorScheduling} is tightly related to the anatomy of its constraints and objective. Indeed, even obtaining a feasible schedules for~\eqref{P:actuatorScheduling} can be hard depending on the structure of the~$\calI_p$. Not to mention obtaining a good one. In fact, \eqref{P:actuatorScheduling} is generally NP-hard~\cite{Krause08n, Das11s, Sagnol13a, Ranieri14n, Tzoumas16m, Zhang17s, Olshevsky18o}. Nevertheless, depending on the nature of the objective and constraints, there exist simple algorithms able to provide approximate solutions. To determine when and how this is possible, we study the generic set function optimization problem
\begin{prob}\label{P:generic}
	\calX^\star \in \argmin_{\calX \subseteq \calE}&
		&&f(\calX)
	\\
	\subjectto& &&\calX \in \calI
		\text{,}
\end{prob}
of which~\eqref{P:actuatorScheduling} is an instance. To see this is the case, let~$\calE = \calVbar$, $f(\calX) = J(\calX)$, and~$\calI = \bigcap_{p = 1}^P \calI_p$. In what follows, we introduce conditions on the constraints~(Section~\ref{S:matroid}) and objective~(Section~\ref{S:alphaSM}) of~\eqref{P:generic} such that a simple algorithm~(greedy search) can be used to generate feasible solutions that we then prove are near-optimal~(Section~\ref{S:nearOptimal}). We then proceed to show that~\eqref{P:actuatorScheduling} satisfies these conditions and derive explicit certificates for greedy control scheduling.

\subsection{Matroids}
\label{S:matroid}

We begin by discussing a constraint structure for which feasible schedules can be easily constructed. Explicitly, we restrict~$\calI$ to be an intersection of independent sets of matroids \blue{defined over the same ground set}. Matroids are entities that extend the notion of linear independence in vector spaces to arbitrary algebraic structures. Formally,

\begin{definition}\label{D:Matroid}
	A \emph{matroid}~$M = \left(\calE, \calI\right)$ consists of a finite set of elements~$\calE$ and a family~$\calI \subseteq 2^\calE$ of subsets of~$\calE$ called \emph{independent sets} that satisfy:
	\begin{enumerate}[P.1]
	\item $\emptyset \in \calI$;

	\item if $\calA \subseteq \calB$ and~$\calB \in \calI$, then $\calA \in \calI$;

	\item if~$\calA,\calB \in \calI$ and~$\abs{\calA} < \abs{\calB}$, then there exists~$e \in \calB \setminus \calA$ such that~$\calA \cup \{e\} \in \calI$.
	\end{enumerate}
\end{definition}

\noindent Hence, we define the collection of valid schedules~$\calI$ in~\eqref{P:generic} to be of the form
\begin{equation}\label{E:matroidIntersection}
	\calI = \bigcap_{p = 1}^P \calI_p
		\text{, such that } (\calE,\calI_p) \text{ is a matroid for all $p$.}
\end{equation}
As we discuss in Section~\ref{S:alphaLQR}, typical scheduling constraints~(in particular those listed in Section~\ref{S:problem}) can be described in terms of independent sets and their intersection. It is worth noting that matroids are not closed under intersections, so that~$\calI$ in~\eqref{E:matroidIntersection} need not be a matroid~\cite[Ch.~39]{Schrijver03c}.

The use of matroid intersections as constraints in~\eqref{P:generic} is attractive for two reasons. The first is a direct consequence of the inheritance property of matroids, i.e., P.1 and~P.2 in Definition~\ref{D:Matroid}.

\begin{proposition}\label{T:inheritanceIntersection}
For each~$\calA \in \calI$ with~$\calI$ as in~\eqref{E:matroidIntersection}, there exists a chain~$\emptyset = \calA_0 \subset \calA_1 \subset \dots \subset \calA_T = \calA$ such that~$\calA_t \in \calI$ for all~$t$. In particular, there exists a chain with~$T = \abs{\calA}$.
\end{proposition}

\noindent Hence, every feasible schedule in~\eqref{P:generic} can be constructed element-by-element. In particular, so can any optimal solution~$\calX^\star$. This suggests an ``interior point''-type algorithm that greedily minimizes the objective~$f$ at each step while maintaining feasibility. Explicitly, a greedy solution of~\eqref{P:generic} is constructed by taking~$\calG_0 = \emptyset$ and incorporating elements of~$\calE$ one at a time, so that at step~$t$,
\begin{equation}\label{E:greedy}
	\calG_{t+1} = \calG_{t} \cup \{g_t\}
\end{equation}
where
\begin{equation}\label{E:greedyElement}
\begin{aligned}
	g_t \in \argmin_{e \in \calE \setminus \calG_{t}}&
		&&f\left( \calG_{t} \cup \{e\} \right)
	\\
	\subjectto& &&\calG_{t} \cup \{e\} \in \calI
		\text{.}
\end{aligned}
\end{equation}
The algorithm stops once no element can be added to~$\calG_t$ without violating feasibility, i.e., when the~$\argmin$ set in~\eqref{E:greedyElement} is empty. Denote that final iteration by~$T$. Naturally, $T \leq \abs{\calE}$ and the algorithm does terminate. What is more, note from~\eqref{E:greedy} that not only is~$\calG_T \in \calI$ by construction, but due to Proposition~\ref{T:inheritanceIntersection}, every set~$\calA \in \calI$ can be constructed by~\eqref{E:greedy} for an appropriate~$f$~(e.g., $f(\calX) = \abs{\calX \cap \calA}$). In other words, the greedy algorithm does not prune any solution from the feasibility set of~\eqref{P:generic}.

The second reason for using matroid intersections is related to the exchange property~(P.3 in Definition~\ref{D:Matroid}) and is laid out in the following proposition~\cite[Sec.~41.1a]{Schrijver03c}:

\begin{proposition}\label{T:exchangeIntersection}
Let~$\calA, \calB \in \calI$ with~$\calI$ as in~\eqref{E:matroidIntersection}. If~$\abs{\calB} > P \abs{\calA}$, then there exist at least~$\abs{\calB} - P \abs{\calA}$ elements~$b$ of~$\calB \setminus \calA$ such that~$\calA \cup \{b\} \in \calI$.
\end{proposition}

\begin{proof}
The proof proceed by induction over the matroids in~$\calI$. The base case for the first matroid~($p = 1$) is readily obtained from P.3 in Definition~\ref{D:Matroid}. Indeed, take~$e \in \calB \setminus \calA$ such that~$\calA \cup \{e\} \in \calI_1$ and let~$\calB^\prime = \calB \setminus \{e\}$. This pruning process can be repeated as long as~$\abs{\calB^\prime} > \abs{\calA}$, i.e., at least~$\abs{\calB} - \abs{\calA}$ times.

Now, suppose the claim holds for the first~$P^\prime-1 < P$ independent sets, i.e., there exists a set~$\calC \subseteq \calB \setminus \calA$ such that
\begin{equation}\label{E:exchangeInduction}
\begin{gathered}
	\abs{\calC} > \abs{\calB} - (P^\prime - 1)\abs{\calA}
	\text{ and }
	\\
	\calA \cup \{c\} \in \bigcap_{p = 1}^{P^\prime-1} \calI_p
	\text{ for all~$c \in \calC$.}
\end{gathered}
\end{equation}
Notice that~$\calB \in \calI \Rightarrow \calB \in \bigcap_{p = 1}^{P^\prime} \calI_p$ and since~$\calC \subset \calB$, the inheritance property of matroids~(P.2 in Definition~\ref{D:Matroid}) implies that~$\calC \in \bigcap_{p = 1}^{P^\prime} \calI_p$ as well. In particular, since~$\calC \in \calI_{P^\prime}$, we again obtain from P.3 in Definition~\ref{D:Matroid} that there exist~$\calC^\prime \subseteq \calC \setminus \calA$ such that
\begin{equation}\label{E:exchangeStep}
	\abs{\calC^\prime} > \abs{\calC} - \abs{\calA}
	\text{ and }
	\calA \cup \{c\} \in \calI_{P^\prime}
		\text{ for all~$c \in \calC^\prime$.}
\end{equation}
Together, \eqref{E:exchangeInduction} and~\eqref{E:exchangeStep} yield that~$\abs{\calC^\prime} > \abs{\calB} - P^\prime\abs{\calA}$ and that~$\calA \cup \{c\} \in \bigcap_{p = 1}^{P^\prime-1} \calI_p \cap \calI_{P^\prime}$ for all~$c \in \calC^\prime$.
\end{proof}

\noindent Though more abstract, this property is fundamental to obtain approximation certificates for the procedure in~\eqref{E:greedy}. Indeed, the following noteworthy corollary is obtained for~$\calB = \calX^\star$ and~$\calA = \calG_t$:

\begin{corollary}
If~$\abs{\calX^\star} > P t$, then there exist at least~$\abs{\calX^\star} - Pt$ elements~$x \in \calX^\star$ such that~$\calG_t \cup \{x\} \in \calI$. Hence, it must be that the algorithm in~\eqref{E:greedy} terminates only after~$T \geq \abs{\calX^\star}/P$.
\end{corollary}

\blue{Note that if~$f$ decreases as the number of elements in~$\calX$ grows, but it does so at a diminishing pace, then we could derive approximation guarantees based on the fact that~\eqref{E:greedy} does not terminate too early. That is where the lower bound on~$T$ is useful. It stems directly from the fact that, if~\eqref{E:greedy} stops and returns a set with~$\abs{\calG_T} < \abs{\calX^\star}/P$, Proposition~\ref{T:exchangeIntersection} implies there exists at least one element~$x \in \calX^\star$ such that~$\calG_T \cup \{x\} \in \calI$. This contradicts the fact that the algorithm terminated.}

Naturally, near-optimality depends not only on the feasibility set~$\calI$, but also on the objective~$f$. For instance, if~$f$ is a monotone decreasing modular function, then~$\calG_T$ is an optimal solution of~\eqref{P:generic} with~$P = 1$ in~\eqref{E:matroidIntersection}~\cite[Ch.~40]{Schrijver03c}; if~$f$ is a monotone decreasing supermodular function, then~$\calG_T$ would be~$1/(1+P)$-optimal~\cite{Fisher78a}. It is well-known, however, that the cost function~$J$ of the original control scheduling problem~\eqref{P:actuatorScheduling} is not modular and is supermodular only under strict conditions~\cite{Olshevsky18o, Tzoumas16m, Singh17s}. We therefore consider a broader class of objective functions~$f$ known as~$\alpha$-supermodular.

\subsection{$\alpha$-supermodularity}
\label{S:alphaSM}

\emph{Supermodularity}~(\emph{submodularity}) refers a ``diminishing returns'' property of certain set functions that yields approximation certificates for their greedy minimization~(maximization). It holds, for instance, for the rank or~$\log\det$ of the controllability~(observability) Gramian of underactuated~(underobserved) systems as well as the Shannon entropy and mutual information of a set of random variables~\cite{Krause14s, Bach13l}. \blue{Note that by \emph{diminishing returns} we mean that the set function decreases~(increases) ever more slowly as the number of elements in the set grows. In that sense, these functions are akin to convex~(concave) functions.}

Supermodularity, however, is a stringent condition that, in particular, does not hold for the LQG objective of the control scheduling problem~\eqref{P:actuatorScheduling}~\cite{Olshevsky18o, Tzoumas16m, Ranieri14n}. To account for this, different forms of \emph{approximate supermodularity}~(\emph{submodularity}) have been proposed by allowing controlled violations of the original ``diminishing returns'' property. The rationale is that if a function is close to supermodular, then it should behave similarly to a truly supermodular function. We formalize these statements in the sequel using the concept of~$\alpha$-supermodularity~\cite{Chamon16n, Chamon17a}.

Let~$f: 2^\calE \to \setR$ be a function defined over subsets~$\calX \subseteq \calE$. We say~$f$ is \emph{normalized} if~$f(\emptyset) = 0$ and~$f$ is \emph{monotone decreasing} if for all sets~$\calA \subseteq \calB \subseteq \calE$ it holds that~$f(\calA) \geq f(\calB)$. Note that if a function is normalized and monotone decreasing, then it is non-positive valued, i.e., $f(\calX) \leq 0$ for all~$\calX \subseteq \calE$. Define the incremental variation
\begin{equation}\label{E:delta}
	\Delta_u f(\calX) = f\left( \calX \right)
		- f\left( \calX \cup \{u\} \right)
\end{equation}
to be the change in the value of~$f$ from adding~$u \in \calE$ to the set~$\calX$. Observe that~$u \in \calX$ implies that~$\Delta_u f(\calX) = 0$. Then,

\begin{definition}\label{D:supermodularity}
A set function~$f: 2^\calE \to \setR$ is \emph{supermodular} if
\begin{equation}\label{E:supermodularity}
	\Delta_u f(\calA) \geq \Delta_u f(\calB)
\end{equation}
for all sets~$\calA \subset \calB \subseteq \calE$ and elements~$u \in \calE \setminus \calB$. A function~$f$ is \emph{submodular} if~$-f$ is supermodular.
\end{definition}

Supermodular, monotone decreasing functions play an important role in the context of scheduling problems due to the celebrated near-optimal guarantees for their greedy minimization. For instance, if~$\calI$ is a single uniform matroid, i.e., $\calI = \{\calS \subseteq \calVbar \mid \abs{\calS} \leq s\}$~(cardinality constraint), then~$\calG_T$ is~$(1-1/e)$-optimal~\cite{Nemhauser78a}. In the general case of~$\calI$ as in~\eqref{E:matroidIntersection}, it holds that~\cite{Fisher78a}
\begin{equation}\label{E:greedySubmodularMatroids}
	f(\calG_T) \leq \frac{1}{1+P} f(\calX^\star)
		\text{.}
\end{equation}
Note that the inequality in~\eqref{E:greedySubmodularMatroids} is reversed compared to Remark~\ref{R:near_optimal} because~$f(\calX) \leq 0$ for all~$\calX$. Without the supermodularity of the objective, solving or even approximating simple instances of~\eqref{P:generic} can be quite challenging: there exist functions~$\delta$-close to supermodular that cannot be approximately optimized in polynomial time unless~$\delta$ is small~\cite{Horel16m}. Hence, \emph{how} we allow this ``diminishing returns'' assumption to be violated is crucial in obtaining useful guarantees. To that effect, we introduce~$\alpha$-supermodularity. This concept was first proposed in the context of auction design~\cite{Lehmann06c}, though it has since been rediscovered and deployed in discrete optimization, estimation, and control~\cite{Sviridenko14o, Chamon16n, Chamon17a, Chamon20a, Bian17g, Summers19p}.

\begin{definition}[$\alpha$-supermodularity]
	\label{D:alphaSM}

A set function~$f: 2^\calE \to \setR$ is \emph{$\alpha$-supermodular}, $\alpha \in \setR_+$, if
\begin{equation}\label{E:alphaSM}
	\Delta_u f(\calA) \geq \alpha \, \Delta_u f(\calB)
		\text{.}
\end{equation}
for all sets~$\calA \subset \calB \subseteq \calE$ and all~$u \in \calE \setminus \calB$. A function~$f$ is \emph{$\alpha$-submodular} if~$-f$ is $\alpha$-supermodular.

\end{definition}

Notice that~$\alpha$-supermodularity is a hereditary property in the sense that any~$\alpha$-supermodular function is also~$\alpha^\prime$-supermodular for all~$\alpha^\prime < \alpha$. What is more, \eqref{E:alphaSM} always holds for~$\alpha = 0$ if~$f$ is monotone decreasing. For that reason, we are often interested in the largest value of~$\alpha$ for which~\eqref{E:alphaSM} holds, i.e.,
\begin{equation}\label{E:alpha}
	\bar{\alpha} =
	\min_{\substack{\calA \subset \calB \subseteq \calE \\
		u \in \calE \setminus \calB}}\
		\frac{\Delta_u f(\calA)}{\Delta_u f(\calB)}
		\text{.}
\end{equation}
When~$\bar{\alpha} \in \left(0,1\right)$, we say~$f$ is \emph{approximately supermodular}. If~$\bar{\alpha} \geq 1$, \eqref{E:alphaSM} implies~\eqref{E:supermodularity}, in which case we refer to~$f$ simply as \emph{supermodular}~\cite{Krause14s, Bach13l}. Thus, $\alpha$-supermodularity can be used to both relax or strengthen the classical concept of supermodularity in Definition~\ref{D:supermodularity}.

Although Definition~\ref{D:alphaSM} is in terms of singleton updates, the approximate diminishing return property in~\eqref{E:alphaSM} also holds for set-valued updates. This equivalent definition is often convenient in derivations.

\begin{proposition}
The function~$f: 2^\calE \to \setR$ is an $\alpha$-supermodular set function if and only if
\begin{equation}\label{E:genAlphaSM}
	f\left(\calY\right) - f\left(\calY \cup \calC\right)
		\geq \alpha \left[ f\left(\calZ\right) - f\left(\calZ \cup \calC\right) \right]
\end{equation}
for all sets~$\calY \subset \calZ \subseteq \calE$ and~$\calC \in \calE \setminus \calZ$.
\end{proposition}

\begin{proof}
The necessity part is straightforward: taking~$\calC = \{u\}$ for~$u \in \calE \setminus \calZ$ in~\eqref{E:genAlphaSM} yields~\eqref{E:alphaSM}. Sufficiency follows by induction. Consider an arbitrary enumeration of~$\calC$ and let~$\calC_t = \{u_1, \dots, u_t\}$ be the set containing its first~$t$ elements. Then, the base case~$\calC_1 = \{u_1\}$ holds directly from the definition of $\alpha$-supermodularity in~\eqref{E:alphaSM}. Suppose now that~\eqref{E:genAlphaSM} holds for~$\calC_{t-1}$, i.e.,
\begin{equation}\label{E:genAlphaSM1}
	f\left(\calY\right) - f\left(\calY \cup \calC_{t-1}\right) \geq
		\alpha \left[  f\left(\calZ\right) - f\left(\calZ \cup \calC_{t-1}\right)\right]
		\text{.}
\end{equation}
Let~$\calA = \calY \cup \calC_{t-1}$, $\calB = \calZ \cup \calC_{t-1}$, and~$e = u_t$ in~\eqref{E:alphaSM} to get
\begin{multline}\label{E:genAlphaSM2}
	f\left(\calY \cup \calC_{t-1}\right) - f\left(\calY \cup \calC_{t-1} \cup \{u_t\}\right) \geq
    \\
	\alpha \left[ f\left(\calZ \cup \calC_{t-1}\right)
		- f\left(\calZ \cup \calC_{t-1} \cup \{u_t\}\right)
	\right]
		\text{.}
\end{multline}
By summing~\eqref{E:genAlphaSM1} and~\eqref{E:genAlphaSM2}, we conclude that~\eqref{E:genAlphaSM} also holds for~$\calC_t = \calC_{t-1} \cup \{u_t\}$.
\end{proof}

\noindent Another important property of~$\alpha$-supermodularity is that it is preserved by non-negative affine combinations.

\begin{proposition}\label{T:preserveSM}
Consider the set functions~$f_i: 2^\calE \to \setR$ with~$f_i$ is~$\alpha_i$-supermodular. Then, for~$\theta_i \geq 0$ and~$\beta \in \setR$, the function~$g = \sum_i \theta_i f_i + \beta$ is~$\min(\alpha_i)$-supermodular.
\end{proposition}

\begin{proof}
See~\cite[Lemma~1]{Chamon20a}
\end{proof}

Definition~\ref{D:alphaSM} turns out to be a fruitful relaxation of Definition~\ref{D:supermodularity}. For instance, it has been shown that if~$f$ is~$\alpha$-supermodular and~$\calI$ is composed of a single uniform matroid~(cardinality constraint), then its greedy solution is~$(1-e^{-\alpha})$-optimal~\cite{Lehmann06c, Sviridenko14o, Chamon16n, Chamon20a}. When~$\calI$ is a single arbitrary matroid~($P = 1$), the greedy solution is~$\alpha/2$-optimal~\cite{Chamon19m}. In these cases, $\alpha$ not only quantifies the diminishing return violations, but also the loss in near-optimality guarantee due to these violations. Explicit lower bounds on~$\alpha$ have been provided in the context of sampling, experiment design, estimation, and control~\cite{Chamon16n, Chamon17a, Chamon19m, Chamon20a}. Still, these guarantees do not directly extend to the matroid intersection constraint in~\eqref{E:matroidIntersection}. In the next section, we determine how relaxing supermodularity to~$\alpha$-supermodularity affects the guarantee in~\eqref{E:greedySubmodularMatroids}.

Before proceeding, however, it is worth noting that~$\alpha$ in~\eqref{E:alpha} is related to the \emph{supermodularity ratio}, a relaxation based on a different, though equivalent, definition of supermodularity. Explicitly, the supermodularity ratio is defined as the largest~$\gamma$ for which
\begin{equation}\label{E:supermodularityRatio}
	\sum_{w \in \calW} \Delta_w f(\calL)
	\geq
	\gamma \left[ f(\calL) - f(\calL \cup \calW) \right]
		\text{,}
\end{equation}
for all disjoint sets~$\calW,\calL \subseteq \calE$. It was introduced in~\cite{Das11s}%
\footnote{Although~\cite{Das11s} and subsequent literature define~$\gamma$ in terms of submodularity, it can be recovered by taking~$-f$ in~\eqref{E:supermodularityRatio}. Recall that if~$f$ is supermodular, then~$-f$ is submodular.}
in the context of variable selection, although the resulting guarantees depended on sparse eigenvalues that are NP-hard to compute. Explicit~(P-computable) lower bounds on~$\gamma$ were derived in~\cite{Bian17g, Karaca18e, Summers19p}, although it is worth noting that the bounds on~$\alpha$ previously obtained in~\cite{Chamon16n, Chamon17t} also hold for the supermodularity ratio:

\begin{proposition}\label{T:alphaGamma}

Let~$f$ be an~$\alpha$-supermodular and denote its supermodularity ratio by~$\gamma$. Then, $\alpha \leq \gamma$.
\end{proposition}

\begin{proof}
We proceed by showing that~\eqref{E:supermodularityRatio} holds with~$\gamma = \alpha$ for any~$\alpha$-supermodular function. To do so, consider an enumeration of~$\calW = \{w_1,\dots,w_\abs{\calW}\}$ and write
\begin{multline}\label{E:alphaGamma2}
	f\left( \calL \right) - f\left( \calL \cup \calW \right) =
		\sum_{k = 1}^\abs{\calW} \Delta_{w_k} f\left( \calL \cup \{w_1,\dots,w_{k-1}\} \right)
		\text{.}
\end{multline}
Since~$f$ is~$\alpha$-supermodular, we can upper bound each of the increments in~\eqref{E:alphaGamma2} using~\eqref{E:alphaSM} to obtain
\begin{equation}\label{E:alphaGamma1}
	f\left( \calL \right) - f\left( \calL \cup \calW \right)
	\leq
	\alpha^{-1} \sum_{k = 1}^\abs{\calW} \Delta_{w_k} f\left( \calL\right)
		\text{.}
\end{equation}
Comparing~\eqref{E:supermodularityRatio} and~\eqref{E:alphaGamma1} concludes the proof.
\end{proof}

\subsection{Matroid-constrained $\alpha$-supermodular minimization}
\label{S:nearOptimal}

At this point, we have fully specified the optimization problem we wish to solve and the algorithm we will use to do so. Explicitly, we are interested in using the greedy method~\eqref{E:greedy} to approximate the solution of~\eqref{P:generic} when the constraint~$\calI$ is an intersection of matroids as in~\eqref{E:matroidIntersection} and the objective~$f$ is a monotone decreasing, $\alpha$-supermodular function~(Definition~\ref{D:alphaSM}). The main result of this section, presented in Theorem~\ref{T:approx}, provides an approximation guarantee under these conditions.

\begin{theorem}\label{T:approx}

Consider~\eqref{P:generic} with~$\calI$ being an intersection of matroids as in~\eqref{E:matroidIntersection} and~$f$ being a normalized, monotone decreasing, $\alpha$-supermodular function~(i.e., $f(\calA) \leq 0$ for all~$\calA \subseteq \calE$). Let~$\calG_T$ be its greedy solution from~\eqref{E:greedy}. Then,
\begin{equation}\label{E:nearOptimality}
	f(\calG_T) \leq \frac{\alpha}{\alpha + P} \, f(\calX^\star)
		\text{.}
\end{equation}
\end{theorem}

Theorem~\ref{T:approx} provides a near-optimal certificate for the greedy solution of $\alpha$-supermodular minimization problems subject to multiple matroid constraints. Since the values of~$f$ are non-positive due to the normalization assumption, the result in~\eqref{E:nearOptimality} may not be straightforward to interpret. Nevertheless, it can be written equivalently in terms of improvements with respect to the empty solution, namely,
\begin{equation*}
	\frac{f(\calG) - f(\calX^\star)}{f(\emptyset) - f(\calX^\star)}
		\leq \frac{P}{\alpha + P}
		\text{.}
\end{equation*}

As expected from previous results, the guarantee in~\eqref{E:nearOptimality} bounds the suboptimality of greedy search in terms of the~$\alpha$-supermodularity of the cost function. When~$\alpha < 1$, \eqref{E:nearOptimality} quantifies the loss in performance guarantee due to the objective~$f$ violating the diminishing returns property. Confirming our initial intuition, the larger the violations, i.e., the smaller~$\alpha$, the worse the guarantee. On the other hand, \eqref{E:nearOptimality} shows that the classical certificate for supermodular functions in~\eqref{E:greedySubmodularMatroids} can be strengthened when~$f$ has stronger diminishing returns structures, i.e., when~$\alpha \geq 1$. It is easy to see that Theorem~\ref{T:approx} is consistent with previous results for single~\cite{Nemhauser78a} and multiple~\cite{Fisher78a} matroid constraints. Indeed, we can recover~\eqref{E:greedySubmodularMatroids} by taking~$\alpha = 1$ in~\eqref{E:nearOptimality}.

Observe, however, that~\eqref{E:nearOptimality} decreases linearly with~$P$. In other words, the guarantees for greedy search deteriorates as more matroids are needed to represent~$\calI$. In a sense, this is the constraint counterpart of~$\alpha$-supermodularity: the further away from a pure matroid the constraints are, the worst the greedy algorithm is guaranteed to perform. Effectively, \eqref{E:nearOptimality} states that the more constraints need to be satisfied, i.e., the larger~$P$, the harder the problem becomes. Note that~$P$ can be replaced by the minimum number of matroids needed to represent~$\calI$, so that the guarantee from Corollary~\ref{T:guarantee} can sometimes be improved. Determining this minimum number, however, can be quite intricate. Still, Theorem~\ref{T:approx} is a \emph{worst case} bound and as is the case with the classical result for greedy supermodular minimization~\cite{Nemhauser78a, Fisher78a}, better performance is often obtained in practice.

It is worth noting that in the single matroid case~($P = 1$), Theorem~\ref{T:approx} is strictly~(though not significantly) stronger than the guarantee from~\cite{Chamon19m} for~$\alpha < 1$. This difference is more accentuated for small~$\alpha$: for instance, if~$\alpha = 0.2$, \eqref{E:nearOptimality} is~$60\%$ stronger than the certificate in~\cite{Chamon19m}. A similar certificate was obtained in~\cite{Lehmann06c} with~$\alpha \geq 1$ for a single matroid constraints.

We now proceed with the proof of Theorem~\ref{T:approx}.

\begin{proof}
This proof follows similarly to the one for the cardinality constraint~(uniform matroid) problem, but relying on the exchange property of matroids through Proposition~\ref{T:exchangeIntersection}. It may be informative to recall the proof technique in that simpler case first, e.g., by referring to~\cite[Thm.~1]{Chamon20a}.

Let~$\abs{\calX^\star} = r$. Partition the optimal solution in~\eqref{P:generic} such that~$\calX^\star = \bigcup_{j = 0}^{\floor{r/P}} \calX^\star_j$, $\abs{\calX^\star_j} \leq P$, and
\begin{equation}\label{E:greedyImproves}
	\calG_t \cup \{x^\star\} \in \calI \text{ for } x^\star \in \calX^\star_t
		\text{ for all } t = 0,\dots,\floor{\frac{r}{P}}
		\text{.}
\end{equation}
Observe that such a partition exists due to the matroid exchange property from Proposition~\ref{T:exchangeIntersection}. Fix an arbitrary enumeration of each partition~$\calX^\star_t = \{x^\star_{t1}, \dots, x^\star_{tp}\}$. The following proof is independent of the specific enumeration.

Use the fact that~$f$ is monotone decreasing to write
\begin{equation}\label{E:telescopic}
\begin{aligned}
	f(\calX^{\star}) &\geq f(\calG_T \cup \calX^{\star})
	\\
	{}&= f(\calG_T) - \sum_{t = 0}^{\floor{r/P}}
		\left[ \sum_{j = 1}^{\abs{\calX^\star_t}} \Delta_{x_{tj}^\star} f(\calT^{j-1}_{t}) \right]
	\text{,}
\end{aligned}
\end{equation}
where the set~$\calT_t^j$ contains the greedy solution~$\calG_T$, the elements of all partitions~$\calX^\star_i$ up to~$t-1$, and the first~$j$ elements of partition~$\calX^\star_t$. They can be defined recursively as~$\calT^0_{0} = \calG_T$, $\calT^0_{t} = \calT^{0}_{t-1} \cup \calX^\star_t$, and~$\calT^j_{t} = \calT^{0}_{t-1} \cup \{x^\star_{t1}, \dots, x^\star_{t j}\}$. Then, we can bound~\eqref{E:telescopic} by distinguishing two cases: (i)~if~$x^\star_{tj} \notin \calT_t^{j-1}$, the~$\alpha$-supermodularity of~$f$ yields
\begin{equation}\label{E:deltaBound}
	\Delta_{x^\star_{tj}} f\left( \calT^{j-1}_{t} \right)
		\leq \alpha^{-1} \Delta_{x^\star_{tj}} f\left( \calG_{t} \right)
		\text{,}
\end{equation}
since~$\calG_{t} \subseteq \calG_T \subseteq \calT^j_{t}$ for all~$t$ and~$j$; (ii)~if~$x^\star_{tj} \in \calT^{j-1}_{t}$, then~$\Delta_{x^\star_{tj}} f\left( \calT^{j-1}_{t} \right) = 0$ and~\eqref{E:deltaBound} holds trivially. Using~\eqref{E:deltaBound} in~\eqref{E:telescopic} gives
\begin{equation}\label{E:telescopicBound}
	f(\calX^{\star}) \geq f(\calG_T) - \alpha^{-1} \sum_{t = 0}^{\floor{r/P}}
		\left[ \sum_{j = 1}^{\abs{\calX^\star_t}} \Delta_{x_{tj}^\star} f(\calG_{t}) \right]
		\text{.}
\end{equation}

The bound in~\eqref{E:telescopicBound} can be simplified using the greedy nature of the algorithm in~\eqref{E:greedy}. Indeed, observe that~\eqref{E:greedyImproves} implies that~$\Delta_{x_{tj}^\star} f(\calG_{t}) \leq \Delta_{g_t} f\left( \calG_{t} \right)$ for~$g_t$ is the~$t$-th element selected as in~\eqref{E:greedyElement}. Hence,
\begin{equation}\label{E:telescopicBound2}
	f(\calX^{\star}) \geq f(\calG_T) - \alpha^{-1} \sum_{t = 0}^{\floor{r/P}}
		\left[ \sum_{j = 1}^{\abs{\calX^\star_t}} \Delta_{g_{t}} f(\calG_{t}) \right]
		\text{.}
\end{equation}
Using the fact that the increments are always non-negative~[see~\eqref{E:delta}] and since~$\abs{\calX^\star_t} \leq P$, \eqref{E:telescopicBound2} reduces to
\begin{equation}\label{E:telescopicBound3}
	f(\calX^{\star}) \geq f(\calG_T) - \alpha^{-1} P \sum_{t = 0}^{\floor{r/P}}
		\Delta_{g_{t}} f(\calG_{t})
		\text{.}
\end{equation}
To conclude, observe from~\eqref{E:delta} that~$f(\calG_T) \leq f(\calG_t) = -\sum_{t = 0}^{t-1} \Delta_{g_{t}} f(\calG_{t})$ and recall from Proposition~\ref{T:exchangeIntersection} that~$T \geq r/P$. Thus,
\begin{equation*}
	f(\calX^{\star}) \geq (1 + \alpha^{-1} P) f(\calG_T)
		\text{,}
\end{equation*}
which can be rearranged as in~\eqref{E:nearOptimality}.
\end{proof}

\section{NEAR-OPTIMAL INPUT SCHEDULING}
\label{S:alphaLQR}

In the previous section, we established when, how, and how well solutions of the generic problem~\eqref{P:generic} can be approximated. We did so by developing a theory for the greedy optimization of~$\alpha$-supermodular functions subject to multiple matroid constraints. In the sequel, we reconcile this general theory with the control scheduling problem~\eqref{P:actuatorScheduling} we set out to solve in the beginning of this paper.

Start by noticing that the feasible set of~\eqref{P:actuatorScheduling} is exactly the intersection of the~$\calI_p$ and that restricting them to be independent sets of matroids is in fact quite mild. Indeed, all typical schedule constraints discussed so far can be described in those terms. They are in fact instances of uniform, partition, and transversal matroids respectively:
\begin{itemize}
	\item bound on the total number of control actions:\\$\calI = \{\calS \subseteq \calVbar \mid \abs{\calS} \leq s\}$,
	\item bound on the number of inputs used per time instant:\\$\calI = \{\calS \subseteq \calVbar \mid \abs{\calS \cap \calV_k} \leq s_k\}$, and
	\item restriction on the consecutive use of inputs:\\$\calI = \{\calS \subseteq \calVbar \mid v^k \notin \calS \text{ or } v^{k+1} \notin \calS \text{, for } v^k,v^{k+1} \in \calVbar \}$,
\end{itemize}
Under these conditions, \eqref{P:actuatorScheduling} can be solved greedily using Algorithm~\ref{L:greedy}. Note that Algorithm~\ref{L:greedy} does not directly implement~\eqref{E:greedy} for~\eqref{P:actuatorScheduling}. Both procedures are nevertheless equivalent due to the matroidal structure of~$\calI$. Indeed, for any set~$\calA \subset \calB \in \bigcap_{p = 1}^P \calI_p$, if~$\calA \cup \{g\} \notin \bigcap_{p = 1}^P \calI_p$ for~$g \in \calVbar$, then~$\calB \cup \{g\} \notin \bigcap_{p = 1}^P \calI_p$. In other words, if adding an element~$g$ to~$\calS_t$ would make the schedule infeasible, then adding it to~$\calS_{t^\prime}$, $t^\prime > t$, would do the same. As opposed to~\eqref{E:greedy}, Algorithm~\ref{L:greedy} simply discards such elements. Additionally, it omits the constant normalization~$V^\star(\emptyset)$ since it does not depend on the schedule.

\begin{algorithm}[tb]
\caption{Greedy algorithm for~\eqref{P:actuatorScheduling}}
	\label{L:greedy}
\setlength{\baselineskip}{1.2\baselineskip}
\begin{algorithmic}
	\State Let~$\calS_0 \leftarrow \emptyset$, $\calZ_0 \leftarrow \calE$, and~$t \leftarrow 0$

	\While{$\calZ_t \neq \emptyset$}
		\State $g \leftarrow \argmin_{u \in \calZ_t} V^\star(\calS_t \cup \{u\})$

		\State $\calZ_{t+1} \leftarrow \calZ_t \setminus \{g\}$

		\If{$\calS_t \cup \{g\} \in \cap_{p = 1}^P \calI_p$}
			\State $\calS_{t+1} \leftarrow \calS_t \cup \{g\}$
		\Else
			\State $\calS_{t+1} \leftarrow \calS_t$
		\EndIf

		\State $t \leftarrow t + 1$
	\EndWhile

	\State $\calS_g \leftarrow \calS_{t}$
\end{algorithmic}
\end{algorithm}

Suffices it now to show that the objective function~$J$ is~$\alpha$-supermodular and monotone decreasing as well as provide an explicit lower bound on~$\bar{\alpha}$ in~\eqref{E:alpha}. Without this bound, the guarantee from Theorem~\ref{T:approx} would be of little practical value as it could not be computed \emph{a priori}. What is more, near-optimality certificates based on~$\alpha$-supermodularity are vacuous unless we can show that~$\alpha$ is strictly larger than zero and, at least under certain conditions, substantially so. We address these issues with the following proposition:

\begin{proposition} \label{T:alphaBound}

Let~$\bA_k$ in~\eqref{E:dynamics} be full rank for all~$k$. The normalized actuator scheduling problem objective~$J$ is (i)~monotonically decreasing and (ii)~$\alpha$-supermodular with
\begin{equation}\label{E:alphaBound}
	\alpha \geq \min_{k = 0,\dots,N-1}\ \alpha_k
		\text{,}
\end{equation}
where
\begin{equation}
	\alpha_k \geq \frac{
		\lambda_{\min} \left[ \bPt_{k+1}^{-1}(\emptyset) \right]
	}{
		\lambda_{\max} \left[ \bPt_{k+1}^{-1}(\calVbar)
			+ \sum_{i \in \calV_k} r_{i,k}^{-1} \bbt_{i,k}^{} \bbt_{i,k}^T \right]
	}
		\text{,}
\end{equation}
for~$\bPt_{k+1}(\calX) = \bH_k^{1/2} \bP_{k+1}(\calX) \bH_k^{1/2}$, $\bbt_{i,k} = \bH_k^{-1/2} \bb_{i,k}$, $\bH_0 = \bA_{0} \bSigma_{0} \bA_{0}^T$, and~$\bH_k = \bA_{k} \bW_{k-1} \bA_{k}^T$ for~$k \geq 1$.

\end{proposition}

\begin{proof}
See appendix.
\end{proof}

\begin{remark}
The full rank hypothesis on~$\bA_k$ can be lifted using a continuity argument. However, the bound in~\eqref{E:alphaBound} is trivial~($\alpha \geq 0$) for rank deficient state transition matrices, so we focus only on the case of practical significance.
\end{remark}

Proposition~\ref{T:alphaBound} states that the objective of~\eqref{P:actuatorScheduling} is~$\alpha$-supermodular for~$\alpha$ as in~\eqref{E:alphaBound}. Notice that the lower bound on~$\bar{\alpha}$ is explicit in that it can be evaluated in terms of the parameters of the dynamical system and the weights~$\bQ_k$ and~$r_{i,k}$ in~\eqref{E:LQG}. \blue{What is more, the complexity of evaluating this bound on~$\alpha$ is essentially that of computing~$\bP_k$ from~\eqref{E:P}, i.e, the complexity of computing the LQG/LQR controller.} In other words, \eqref{E:alphaBound} allows us to efficiently determine~$\alpha$ for the objective~$J$ \emph{a priori}. Combining Proposition~\ref{T:alphaBound} and Theorem~\ref{T:approx} yields the approximation guarantee for the greedy solution of~\eqref{P:actuatorScheduling} collected below.

\begin{corollary}\label{T:guarantee}

Consider the control scheduling problem~\eqref{P:actuatorScheduling} in which~$(\calVbar,\calI_p)$ is a matroid for each~$p$ and let~$\calS^\star$ be its optimal solution and~$\calS_g$ be its greedy solution obtained using Algorithm~\ref{L:greedy}. Then,
\begin{equation}\label{E:schedulingGuarantee}
	J(\calS_g) \leq \frac{\alpha}{\alpha + P} \, J(\calS^\star)
\end{equation}
with~$\alpha \geq \min_{k = 0,\dots,N-1} \alpha_k$ for
\begin{equation}\label{E:alphaBoundGuarantee}
	\alpha_k \geq \frac{
		\lambda_{\min} \left[ \bPt_{k+1}^{-1}(\emptyset) \right]
	}{
		\lambda_{\max} \left[ \bPt_{k+1}^{-1}(\calVbar)
			+ \sum_{i \in \calV_k} r_{i,k}^{-1} \bbt_{i,k}^{} \bbt_{i,k}^T \right]
	}
		\text{,}
\end{equation}
where~$\bPt_{k+1}(\calX) = \bH_k^{1/2} \bP_{k+1}(\calX) \bH_k^{1/2}$, for~$\bP_{k}$ defined in~\eqref{E:P}, and~$\bbt_{i,k} = \bH_k^{-1/2} \bb_{i,k}$. The transformations are the positive-definite square roots of~$\bH_k$, defined as~$\bH_0 = \bA_{0} \bSigma_{0} \bA_{0}^T$ and~$\bH_k = \bA_{k} \bW_{k-1} \bA_{k}^T$ for~$k \geq 1$.

\end{corollary}

Corollary~\ref{T:guarantee} provides a near-optimal certificate for the greedy solution of the control scheduling problem~\eqref{P:actuatorScheduling} as a function of its parameters. Note that the larger the~$\alpha$, the better the guarantee, and that although~$J$ is not supermodular in general, there are situations in which its violations of the diminishing returns property are small.

To see when this is the case, observe that~\eqref{E:alphaBoundGuarantee} is large when
\begin{equation}\label{E:alphaConditions}
	\sum_{i \in \calV_0} r_{i,0}^{-1} \bbt_i \bbt_i^T \text{ is small}
	\quad \text{and} \quad
	\bPt_1^{-1}\left( \calVbar \right) \approx \bPt_1^{-1}(\emptyset)
		\text{.}
\end{equation}
Conditions in~\eqref{E:alphaConditions} occur when~$\diag\left( r_{i,k} \right) \ggcurly \bQ_k$, i.e., when the controller~\eqref{E:LQG} gives more weight to the input cost than state regulation. This condition is readily obtained from the definition of~$\bP_k$ in~\eqref{E:P}. Figures~\ref{F:alphabound_lqr} and~\ref{F:alphabound_lqg} illustrates these observations by evaluating the bound from~\eqref{E:alphaBoundGuarantee} for $100$ random systems with~$n = 100$ states controlled over a horizon of~$N = 4$ steps with~$\bQ = \bI$ and~$r_{i,k} = \gamma$~(see Section~\ref{S:sims} for details). Clearly, as the controller actions cost grows, i.e., as $\gamma$ increases, $\alpha$ grows and Corollary~\ref{T:guarantee} yields stronger guarantees. Note that this is a scenario of considerable practical value. It is well-known that if no restriction is imposed on the input cost, any controllable set of inputs can drive the system to any state in a single step~(\emph{dead-beat controller}). Hence, when~$\diag\left( r_{i,k} \right) \llcurly \bQ_k$, the optimal value of the LQG only really differentiates between schedules that lead to controllable and uncontrollable input sets. On the other hand, careful scheduling can have a considerable impact when the actuation~(or agent deployment) costs are high. This is also the case in which Corollary~\ref{T:guarantee} provides stronger guarantees.

Observe also that since matrices are weighted by~$\bH_k$, the condition numbers of~$\bA$, $\bSigma_0$, and~$\bW_k$ play an important role in the guarantees. Indeed, if~$\bH_k$ is poorly conditioned, there may be a large difference between the minimum and maximum eigenvalues in~\eqref{E:alphaBoundGuarantee}. This is illustrated in Figure~\ref{F:alphabound_lqr}, which shows that when the decay rate of the system modes are considerably different, i.e., the state transition matrix has large condition number~$\kappa(\bA)$, the guarantees from Corollary~\ref{T:guarantee} worsen. Figure~\ref{F:alphabound_lqg} illustrates this phenomenon for the LQG controller, showing how the value of~$\alpha$ decreases when the process noise does not affect the states homogeneously, i.e., the condition number of~$\bW_k$ grows.

In the sequel, we provide details on the experiments presented in this section and illustrate the use of greedy control scheduling under multiple non-trivial constraints using an agent dispatching application.

\begin{figure}[tb]
	\centering
	\includesvg[width=\columnwidth]{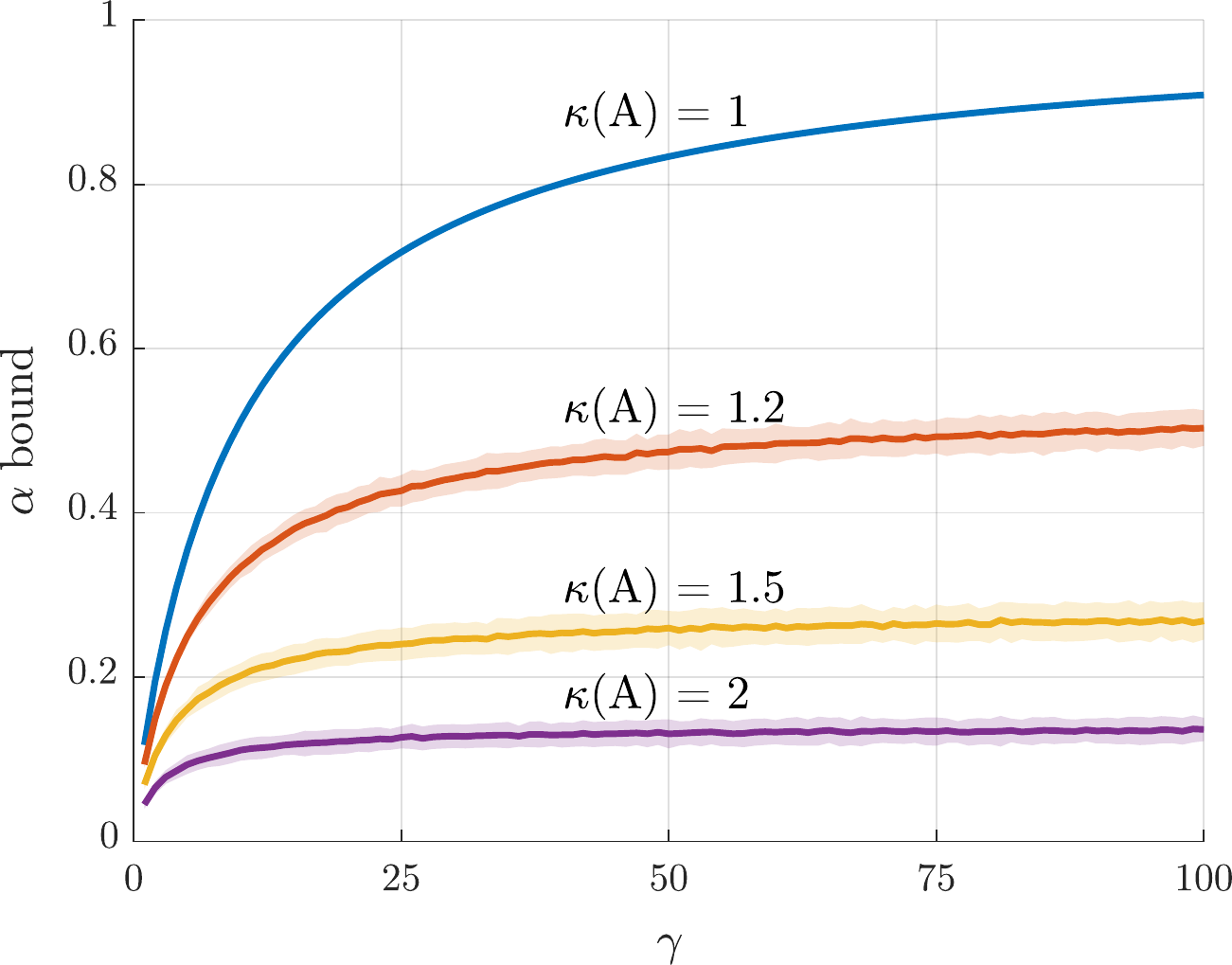}
    \caption{Bound on~$\alpha$ from~\eqref{E:alphaBound} for a deterministic dynamical system~($\bW_k = \bzero$) and a schedule of length~$N = 4$. Shaded regions span two standard deviations from the mean.}
    \label{F:alphabound_lqr}
\end{figure}

\section{NUMERICAL EXAMPLES}
\label{S:sims}

\subsection{Greedy scheduling performance}

We start by detailing the experiments in Figures~\ref{F:alphabound_lqr} and~\ref{F:alphabound_lqg}. We considered dynamical systems with~$n = 100$ states for which the elements of~$\bA$ were drawn randomly from a standard Gaussian distribution, the input matrix~$\bB = \bI$, i.e., direct state actuation, and~$\bPi_0 = 10^{-2} \bI$. The state transition matrix~$\bA$ was normalized so that its spectral radius is~$1.1$, i.e., the dynamical systems are unstable. In Figure~\ref{F:alphabound_lqr}, the dynamical systems are deterministic, i.e., $\bW_k = \bzero$. In Figure~\ref{F:alphabound_lqg}, $\bW_k$ is a diagonal matrix whose elements were drawn uniformly at random from~$[\kappa(\bW)^{-1}, 1]$ so that their condition number is~$\kappa(\bW)$. The figures show the result for~$100$ independently drawn dynamical systems considering~\eqref{P:actuatorScheduling} for~$\bQ = \bI$, $r_{i,k} = \gamma$ for all~$i$ and~$k$, and~$N = 4$ steps. \blue{In these experiments, we use a short horizon so that we are later able to compare our guarantees with the exact approximation ratios achieved by Algorithm~\ref{L:greedy}~(Figure~\ref{F:hist}), which requires computing the optimal schedule by exhaustive search.}

\begin{figure}[tb]
	\centering
	\includesvg[width=\columnwidth]{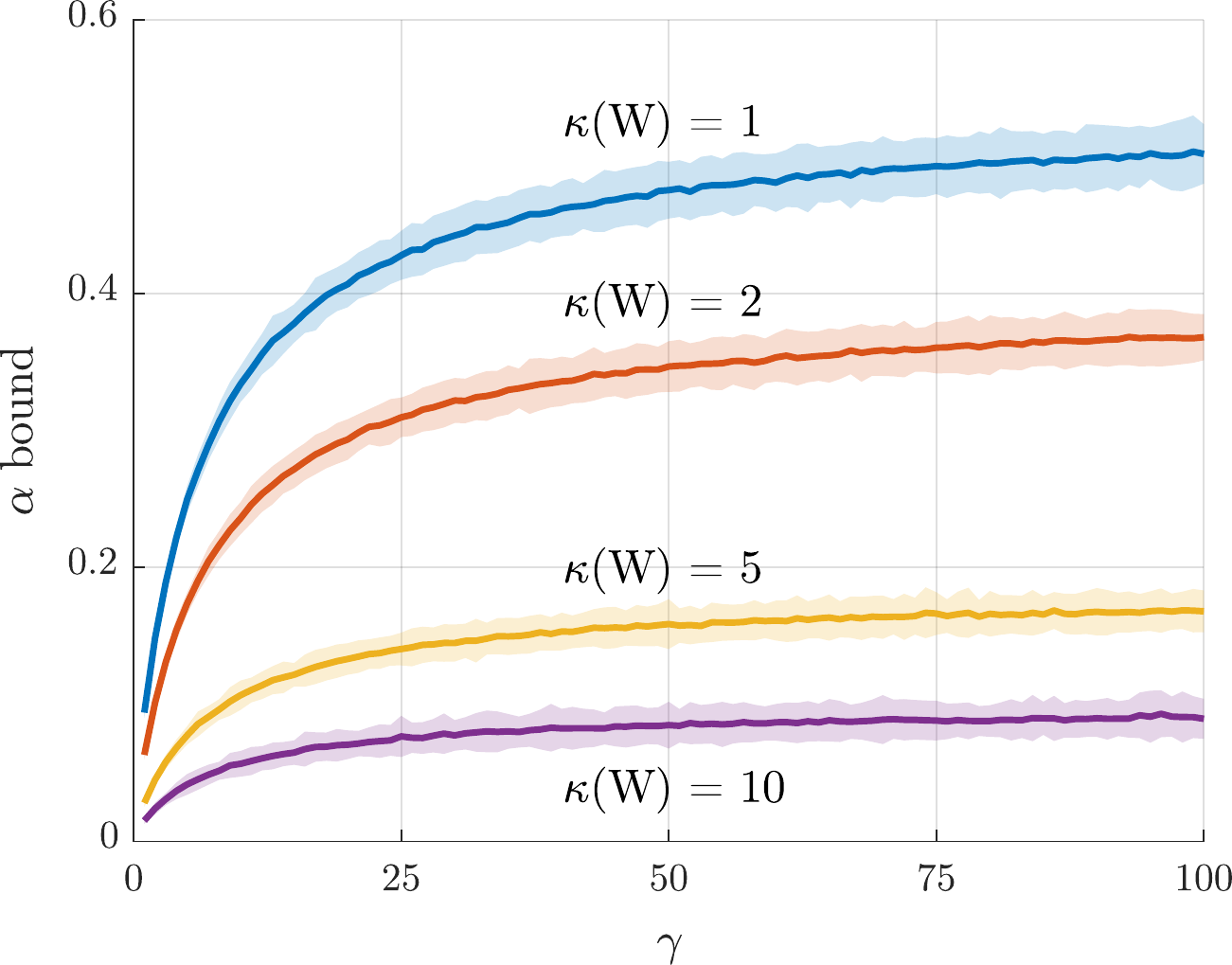}
    \caption{Bound on~$\alpha$ from~\eqref{E:alphaBound} for a dynamical system with disturbance and a schedule of length~$N = 4$. Shaded regions span two standard deviations from the mean.}
    \label{F:alphabound_lqg}
\end{figure}

\begin{figure*}[t]
\begin{minipage}[c]{0.32\textwidth}
\centering
\includesvg[width=\columnwidth]{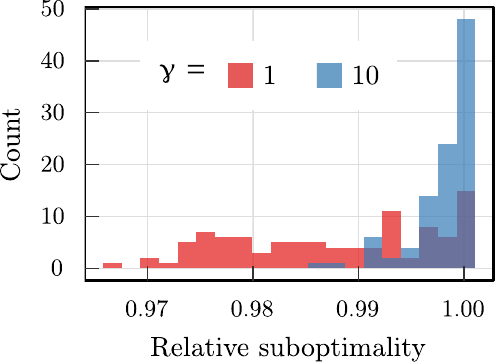}

{\small (a)~$P = 1$}
\end{minipage}
\hfill
\begin{minipage}[c]{0.32\textwidth}
\centering
\includesvg[width=\columnwidth]{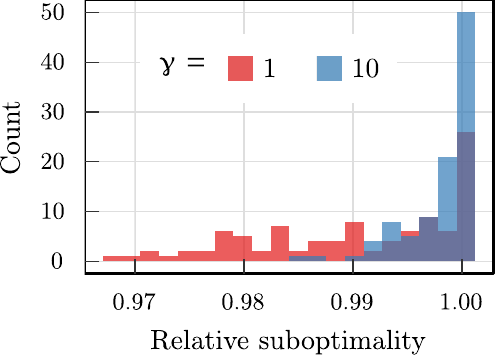}

{\small (b)~$P = 2$}
\end{minipage}
\hfill
\begin{minipage}[c]{0.32\textwidth}
\centering
\includesvg[width=\columnwidth]{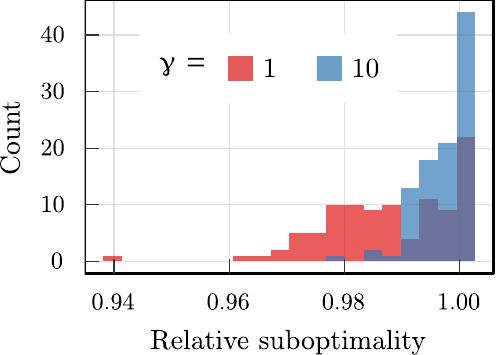}

{\small (c)~$P = 3$}
\end{minipage}
	\caption{Relative suboptimality of greedy scheduling for different constraints~($100$ system realizations). (a)~$\calI_1$~(less than~$2$ actuators per time step); (b)~$\calI_1$ and $\calI_2$~(less than~$5$ control actions over the horizon); (c) $\calI_1$, $\calI_2$, and $\calI_3$~(inputs cannot be used on consecutive time slots).}
	\label{F:hist}
\end{figure*}

While Figures~\ref{F:alphabound_lqr} and~\ref{F:alphabound_lqg}, along with Corollary~\ref{T:guarantee}, illustrate the wide range of parameters over which good performance certificates can be provided, it is worth noting that these are worst-case guarantees and that better results are common in practice. To illustrate this point, we evaluate the relative suboptimality of greedily selected schedules over~$100$ system realizations. Explicitly, we evaluate
\begin{equation*}
	\nu^\star(\calS_g) = \frac{J(\calS_g)}{J(\calS^\star)}
		\text{,}
\end{equation*}
where~$\calS_g$ and~$\calS^\star$ are the greedy~(Algorithm~\ref{L:greedy}) and optimal solutions of~\eqref{P:actuatorScheduling} respectively. Since~$\nu^\star$ depends on the optimal schedule, which can only be obtained by exhaustive search, we restrict ourselves to smaller dynamical systems with $n = 7$~states. State space matrices~$\bA$ and~$\bB$ are as before, with~$\kappa(\bA) = 1.2$, and~$\bW_k = \bzero$~(LQR case). We again take~$\bQ = \bI$ and~$r_{i,k} = \gamma$.

In Figure~\ref{F:hist}, we design schedules for~$N = 4$ steps horizons with different numbers of matroid constraints~$P$. In Figure~\ref{F:hist}a, we select at most~$2$ actuators per time step~($\calI_1 = \{\calS \subseteq \calVbar \mid \abs{\calS \cap \calV_k} \leq 2\}$). In addition to~$\calI_1$, Figure~\ref{F:hist}b also imposes that at most~$5$ control actions can be taken over the horizon~($\calI_2 = \{\calS \subseteq \calVbar \mid \abs{\calS} \leq 5\}$). Finally, Figure~\ref{F:hist}c includes a restriction on using the same input on consecutive time steps on top of~$\calI_1$ and~$\calI_2$. Despite the fact that the lower bounds on~$\nu^\star$~(recall that~$J$ is a non-positive function) from Corollary~\ref{T:guarantee} range from~$0.03$ to~$0.25$, Figure~\ref{F:hist} show that the typical performance of greedy scheduling in practice is considerably better. Though it did not often find the optimal schedule~(between~$15\%$ and~$23\%$ of the realizations for~$\gamma = 1$ and~$40\%$ and~$46\%$ for~$\gamma = 10$), the resulting schedule performances were at most~$5\%$ lower than the optimum. Note, however, that there exist specific dynamical systems for which greedy performs close to the guarantee in Corollary~\ref{T:guarantee}. Indeed, consider~$\bA = \bI$ and
\begin{equation*}
\bB = \begin{bmatrix}
2 & 1 & 0
\\
2 & 0 & 1
\\
1 & 1 & 1
\end{bmatrix}
	\text{.}
\end{equation*}
Under constraint~$\calI_1$, i.e., if we schedule up to~$2$ inputs per time step, over a~$N = 2$ steps horizon with~$\gamma = 100$, the guarantee in~\eqref{E:schedulingGuarantee} yields~$\nu^\star(\calS_g) \geq 0.392$, whereas in practice we achieve~$\nu^\star(\calS_g) \approx 0.423$.

\subsection{Agent dispatching on the Amazon river}

\begin{figure}[tb]
	\centering
	\includesvg{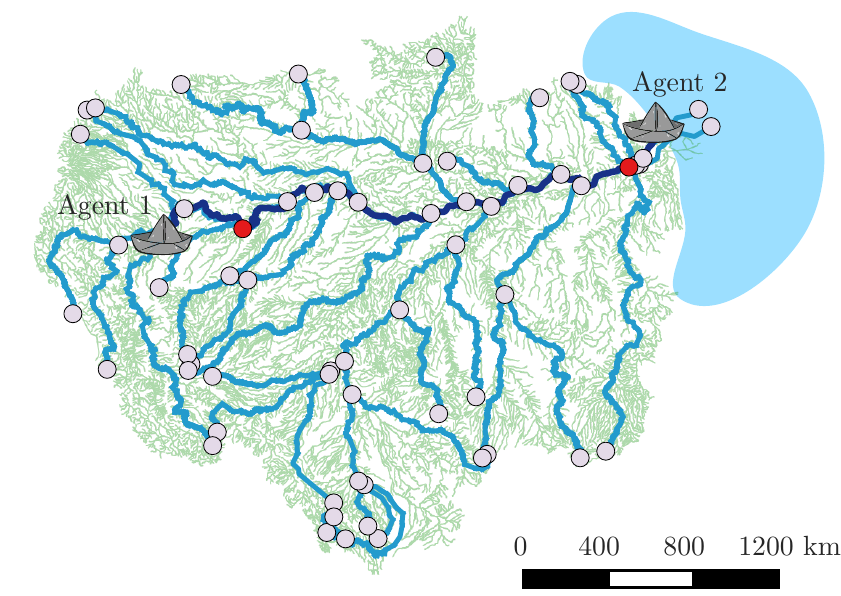}
	\caption{Amazon basin, amazon river~(dark blue trace), system states~(light grey circles), and chemical spill origins~(red circles).}
	\label{F:amazon}
\end{figure}

In this section, we illustrate the use of greedy control scheduling in an agent dispatching application to control the effect of spills on the Amazon river. Two agents navigate up and down the river~(dark blue curve in Figure~\ref{F:amazon}) over a preset route and use a chemical component to counteract damaging spills. Due to the limited capacity of each vessel and to avoid overusage, each agent is allowed to dump the component at most~$5$ times. What is more, at least~$2$ steps must be allowed between each action so the crew has time to setup. The goal of this dispatch is to reduce how much of the spill reaches the ocean~(light blue water mass in Figure~\ref{F:amazon}).

The Amazon drainage basin, showed in green in Figure~\ref{F:amazon}, covers $7.5$~million~$\text{km}^2$ and is composed of over~$7000$ tributaries. We use a simplified description of the basin~(blue traces in Figure~\ref{F:amazon}) obtained by smoothing the original map~\cite{Muller99u}. Using this river network, we construct a weighted directed tree~$\calG$ whose vertices are the circles in Figure~\ref{F:amazon} with the addition of a node midway between each circle. These vertices compose the~$n = 127$ states~$\bx_k$ of the dynamical system. In~$\calG$, two vertices are connected if water flows between them, i.e., if there is a blue line between them in Figure~\ref{F:amazon}. We assume that all river flow toward the ocean, denoted in Figure~\ref{F:amazon} by a light blue water mass. The adjacency matrix of~$\calG$ is then defined as
\begin{equation}\label{E:amazonAdjacency}
	[\bG]_{ij} =
	\begin{cases}
		\norm{\bz_i - \bz_j}
			\text{,} &\text{if } (i,j) \in \calE
		\\
		0
			\text{,} &\text{otherwise}
	\end{cases}
\end{equation}
where~$\bz_i \in \setR^2$ is the position of node~$i$ on the map. We define its Laplacian as~$\bL = \bD - \bG$, where~$\bD$ is a diagonal matrix whose elements are the sum of the columns of~$\bG$, and its symmetrized Laplacian as~$\bL^\prime = \bD^\prime - (\bG + \bG^T)$, where~$\bD^\prime$ is a diagonal matrix whose elements are the sum of the columns of~$\bG + \bG^T$.

\begin{figure}[tb]
	\centering
	\includesvg[width=\columnwidth]{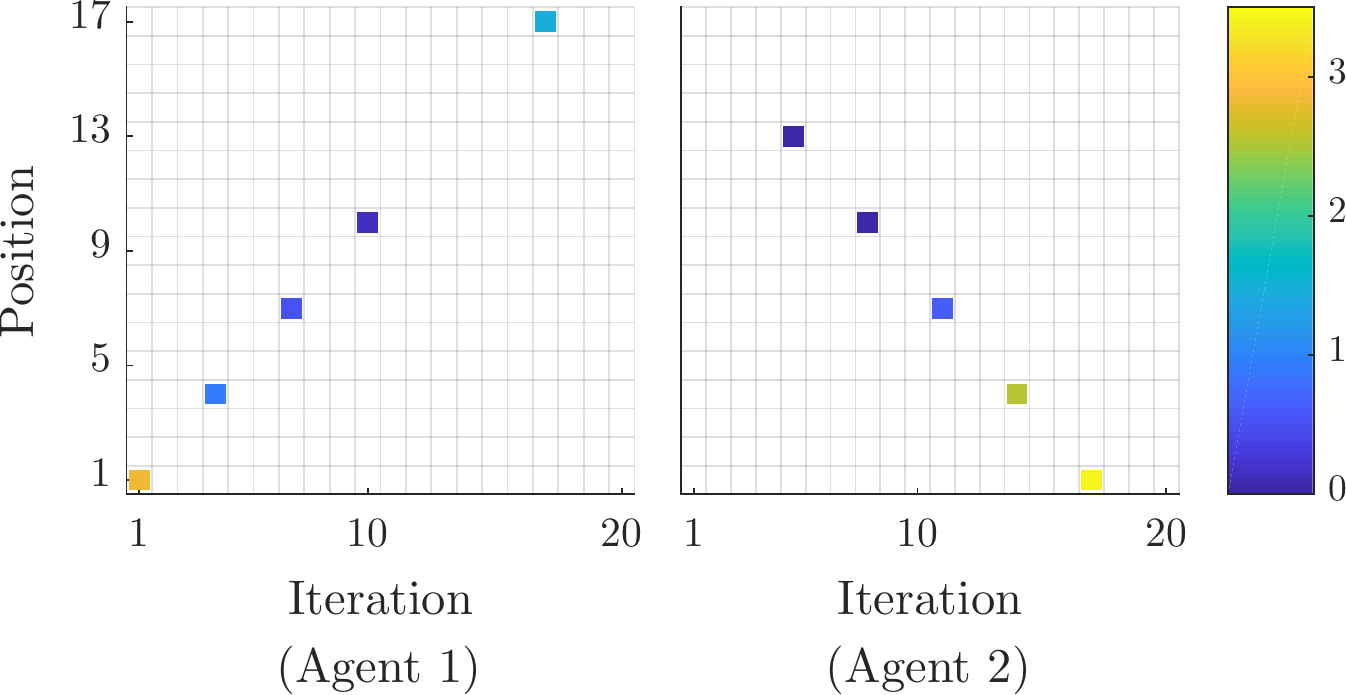}
	\caption{Greedy schedule and actuation energy of the spill control agents.}
	\label{F:amazonSchedule}
\end{figure}

Using these Laplacians, we define the state transition matrix of the dynamical system in~\eqref{E:dynamics} as
\begin{equation}\label{E:amazonDynamics}
	\bA = 0.901 \exp(-\bL \Delta t) + 0.099 \exp(- \bL^\prime \Delta t)
		\text{,}
\end{equation}
where~$\Delta t = 5$ is the sampling period. The dynamics in~\eqref{E:amazonDynamics} are a combination of two processes: the first term corresponds to the \emph{advection} process by which water flows to the ocean and the second term corresponds to a \emph{diffusion} process. The combination coefficients are chosen so that the system is marginally stable~($\norm{\bA} = 1$). In these experiments, we assume there is no process noise, i.e., $\bW_k = \bzero$ for all~$k$, and direct state actuation. The two spills are modeled by spikes in the initial state, namely~$\bx_0$ is zero except at the red nodes in Figure~\ref{F:amazon}.

Agent~$1$ navigates the Amazon river~(dark blue curve) left-to-right and Agent~$2$ navigates right-to-left~(starting near the ocean). They can only actuate on their current position and are only allowed to do so on the states marked as circles in Figure~\ref{F:amazon}. Thus, a centralized greedy scheduler designs a~$N = 20$ steps action plan for the two agents taking into account their positions at each time step and their total number of actions and duty cycle constraints. To do so, it assumes~$\bQ = \bI$ for both agents, but takes~$r_{i,k}^{(1)} = 10$ for Agent~$1$ and~$r_{i,k}^{(2)} = 20$ for Agent~$2$. In other words, Agent~$1$ is allowed to dump more cleaning component. The results of this dispatch are shown in Figures~\ref{F:amazonSchedule} and~\ref{F:amazonStates}.

Notice that in the final schedule~(Figure~\ref{F:amazonSchedule}), the agents effectively act on top of the spills and towards the middle of the river to try and stop their spreading. Agent~$1$, in particular, saves cleaning reagent for one last action next to ocean. What is more, since it pays a lower price for actuation~($r_{i,k}^{(1)} < r_{i,k}^{(2)}$), it is able to use more reagent than Agent~$2$, which ends up concentrating its efforts in the beginning of its route~(Figures~\ref{F:amazonStates}). In the end, the vessels are able to mitigate the impact of the spills on the ocean waters, reducing contamination levels by~$60\%$ compared to the no-actuation solution. The final level achieved with these punctual actions is comparable to using an $64$~agents that actuate every state at every time step with cost matrices~$\bQ = \bI$ and~$r_{i,k} = 2500$~(Figure~\ref{F:amazonStates}).

\begin{figure}[tb]
	\centering
	\includesvg[width=\columnwidth]{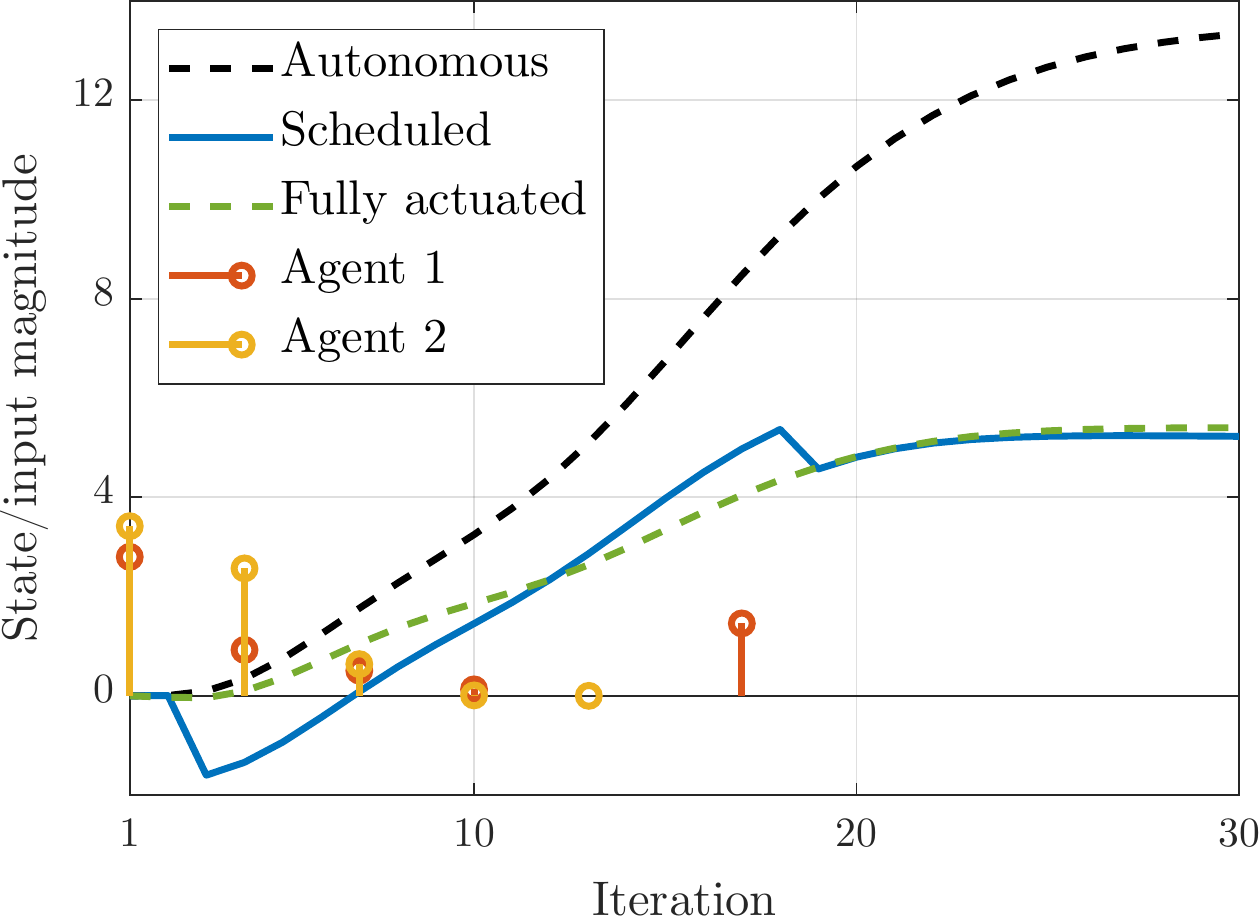}
	\caption{Agents actions and chemical concentrations in the ocean~(light blue mass in Figure~\ref{F:amazon}) for the autonomous system~(no agent), greedy schedule, and full actuation.}
	\label{F:amazonStates}
\end{figure}

\section{CONCLUSION}
\label{S: Conclusion}

We studied the control scheduling problem in which agents or actuators are assigned to act upon a dynamical system at specific times so as to minimize the LQG/LQR objective. We considered the case in which budget and operational scheduling constraints can be encoded as an intersection of $P$~matroids, which can describe any combination of limits on the number of agents deployed per time slot, total number of actuator uses, duty cycle restrictions\dots We showed that this structure allows schedules to be built greedily, i.e., by adding system inputs one-by-one choosing the one that most reduces the control cost. We then showed that for $\alpha$-supermodular objectives this procedure is $\alpha/(\alpha + P)$-optimal and provided an explicit, system-dependent lower bound on~$\alpha$ for the LQG/LQR cost. This bound approaches unity and yield supermodular-like guarantees in relevant application scenarios. Combining these results, we presented an approximation certificate for greedy scheduling that validates its empirical success. We believe that the near-optimality results derived for approximately supermodular functions are of independent interest and can be extended to other applications, such as experimental design using polymatroids. An important future challenge to overcome in these discrete control problems is taking into account state and actuation constraints.

\bibliographystyle{IEEEbib}
\bibliography{IEEEabrv,gsp,sp,math,control,stat}

\begin{thebibliography}{10}

\bibitem{Chamon19m}
L.F.O. Chamon, A.~Amice, and A.~Ribeiro,
\newblock ``Matroid-constrained approximately supermodular optimization for
  near-optimal actuator scheduling,''
\newblock in {\em Conf. on Decision and Contr.}, 2019.

\bibitem{Gerkey04a}
B.P. Gerkey and M.J. Matari\'{c},
\newblock ``A formal analysis and taxonomy of task allocation in multi-robot
  systems,''
\newblock {\em The Int. J. of Robot. Research}, vol. 23[9], pp. 939--954, 2004.

\bibitem{Shamaiah10g}
M.~Shamaiah, S.~Banerjee, and H.~Vikalo,
\newblock ``Greedy sensor selection: {L}everaging submodularity,''
\newblock in {\em Conf. on Decision and Contr.}, 2010, pp. 2572--2577.

\bibitem{Summers16o}
T.H. Summers, F.L. Cortesi, and J.~Lygeros,
\newblock ``On submodularity and controllability in complex dynamical
  networks,''
\newblock {\em {IEEE} Trans. Contr. Netw. Syst.}, vol. 3[1], pp. 91--101, 2016.

\bibitem{Chamon20a}
L.F.O. Chamon, G.J. Pappas, and A.~Ribeiro,
\newblock ``Approximate supermodularity of {K}alman filter sensor selection,''
\newblock {\em {IEEE} Trans. Autom. Control}, vol. 66[1], pp. 49--63, 2021.

\bibitem{Das11s}
A.~Das and D.~Kempe,
\newblock ``Submodular meets spectral: Greedy algorithms for subset selection,
  sparse approximation and dictionary selection,''
\newblock in {\em Int. Conf. on Mach. Learning}, 2011.

\bibitem{Sagnol13a}
G.~Sagnol,
\newblock ``Approximation of a maximum-submodular-coverage problem involving
  spectral functions, with application to experimental designs,''
\newblock {\em Discrete Appl. Math.}, vol. 161[1-2], pp. 258--276, 2013.

\bibitem{Krause08n}
A.~Krause, A.~Singh, and C.~Guestrin,
\newblock ``Near-optimal sensor placements in {G}aussian processes: Theory,
  efficient algorithms and empirical studies,''
\newblock {\em J. Mach. Learning Research}, vol. 9, pp. 235--284, 2008.

\bibitem{Ranieri14n}
J.~Ranieri, A.~Chebira, and M.~Vetterli,
\newblock ``Near-optimal sensor placement for linear inverse problems,''
\newblock {\em {IEEE} Trans. Signal Process.}, vol. 62[5], pp. 1135--1146,
  2014.

\bibitem{Tzoumas16m}
V.~Tzoumas, M.A. Rahimian, G.J. Pappas, and A.~Jadbabaie,
\newblock ``Minimal actuator placement with bounds on control effort,''
\newblock {\em {IEEE} Trans. Contr. Netw. Syst.}, vol. 3[1], pp. 67--78, 2016.

\bibitem{Zhang17s}
H.~Zhang, R.~Ayoub, and S.~Sundaram,
\newblock ``Sensor selection for {K}alman filtering of linear dynamical
  systems: Complexity, limitations and greedy algorithms,''
\newblock {\em Automatica}, vol. 78, pp. 202--210, 2017.

\bibitem{Olshevsky18o}
A.~Olshevsky,
\newblock ``On (non)supermodularity of average control energy,''
\newblock {\em {IEEE} Trans. Contr. Netw. Syst.}, vol. 5[3], pp. 1177--1181,
  2018.

\bibitem{Weimer08a}
J.E. Weimer, B.~Sinopoli, and B.H. Krogh,
\newblock ``A relaxation approach to dynamic sensor selection in large-scale
  wireless networks,''
\newblock in {\em Int. Conf. on Distributed Computing Syst.}, 2008, pp.
  501--506.

\bibitem{Joshi09s}
S.~Joshi and S.~Boyd,
\newblock ``Sensor selection via convex optimization,''
\newblock {\em {IEEE} Trans. Signal Process.}, vol. 57[2], pp. 451--462, 2009.

\bibitem{Dhingra14a}
N.K. Dhingra, M.R. Jovanovi\'{c}, and Z.Q. Luo,
\newblock ``An {ADMM} algorithm for optimal sensor and actuator selection,''
\newblock in {\em Conf. on Decision and Contr.}, 2014, pp. 4039--4044.

\bibitem{Rusu18s}
C.~Rusu, J.~Thompson, and N.M. Robertson,
\newblock ``Sensor scheduling with time, energy, and communication
  constraints,''
\newblock {\em {IEEE} Trans. Signal Process.}, vol. 66[2], pp. 528--539, 2018.

\bibitem{Vitus12o}
M.P. Vitus, W.~Zhang, A.~Abate, J.~Hu, and C.J. Tomlin,
\newblock ``On efficient sensor scheduling for linear dynamical systems,''
\newblock {\em Automatica}, vol. 48[10], pp. 2482--2493, 2012.

\bibitem{Jawaid15s}
S.T. Jawaid and S.L. Smith,
\newblock ``Submodularity and greedy algorithms in sensor scheduling for linear
  dynamical systems,''
\newblock {\em Automatica}, vol. 61, pp. 282--288, 2015.

\bibitem{Singh17s}
P.~Singh, M.~Chen, L.~Carlone, S.~Karaman, E.~Frazzoli, and D.~Hsu,
\newblock ``Supermodular mean squared error minimization for sensor scheduling
  in optimal {K}alman filtering,''
\newblock in {\em American Contr. Conf.}, 2017, pp. 5787--5794.

\bibitem{Jadbabaie18d}
A.~Jadbabaie, A.~Olshevsky, and M.~Siami,
\newblock ``Deterministic and randomized actuator scheduling with guaranteed
  performance bounds,''
\newblock {\em {IEEE} Trans. Autom. Control}, vol. 66[4], pp. 1686--1701, 2021.

\bibitem{Schrijver03c}
A.~Schrijver,
\newblock {\em Combinatorial Optimization},
\newblock Springer, 2003.

\bibitem{Fisher78a}
M.L. Fisher, G.L. Nemhauser, and L.A. Wolsey,
\newblock ``An analysis of approximations for maximizing submodular set
  functions—ii,''
\newblock in {\em Polyhedral combinatorics}, pp. 73--87. Springer, 1978.

\bibitem{Calinescu11m}
G.~Calinescu, C.~Chekuri, M.~P\'{a}l, and J.~Vondr\'{a}k,
\newblock ``Maximizing a monotone submodular function subject to a matroid
  constraint,''
\newblock {\em SIAM J. on Computing}, vol. 40[6], pp. 1740--1766, 2011.

\bibitem{Sviridenko14o}
M.~Sviridenko, J.~Vondr\'{a}k, and J.~Ward,
\newblock ``Optimal approximation for submodular and supermodular optimization
  with bounded curvature,''
\newblock in {\em SIAM Symposium on Discrete Algorithms}, 2014, pp. 1134--1148.

\bibitem{Chamon16n}
L.F.O. Chamon and A.~Ribeiro,
\newblock ``Near-optimality of greedy set selection in the sampling of graph
  signals,''
\newblock in {\em Global Conf. on Signal and Inform. Process.}, 2016, pp.
  1265--1269.

\bibitem{Chamon17a}
L.F.O. Chamon and A.~Ribeiro,
\newblock ``Approximate supermodularity bounds for experimental design,''
\newblock in {\em Advances in Neural Information Processing Systems}, 2017, pp.
  5403--5412.

\bibitem{Bian17g}
A.~Bian, J.M. Buhmann, A.~Krause, and S.~Tschiatschek,
\newblock ``Guarantees for greedy maximization of non-submodular functions with
  applications,''
\newblock in {\em Int. Conf. on Mach. Learning}, 2017.

\bibitem{Summers19p}
T.~Summers and M.~Kamgarpour,
\newblock ``Performance guarantees for greedy maximization of non-submodular
  set functions in systems and control,''
\newblock in {\em European Contr. Conf.}, 2019, pp. 2796--2801.

\bibitem{Korte12c}
B.~Korte and J.~Vygen,
\newblock {\em Combinatorial Optimization: {T}heory and Algorithms},
\newblock Springer, 2012.

\bibitem{Ye18c}
L.~Ye, S.~Roy, and S.~Sundaram,
\newblock ``On the complexity and approximability of optimal sensor selection
  for {K}alman filtering,''
\newblock in {\em American Contr. Conf.}, 2018, pp. 5049--5054.

\bibitem{Bertsekas17d}
D.P. Bertsekas,
\newblock {\em Dynamic Programming and Optimal Control -- {Vol.~I}},
\newblock Athena Scientific, 2017.

\bibitem{Krause14s}
A.~Krause and D.~Golovin,
\newblock ``Submodular function maximization,''
\newblock in {\em Tractability: Practical Approaches to Hard Problems}.
  Cambridge University Press, 2014.

\bibitem{Bach13l}
F.~Bach,
\newblock ``Learning with submodular functions: A convex optimization
  perspective,''
\newblock {\em Foundations and Trends in Machine Learning}, vol. 6[2-3], pp.
  145--373, 2013.

\bibitem{Nemhauser78a}
G.L. Nemhauser, L.A. Wolsey, and M.L. Fisher,
\newblock ``An analysis of approximations for maximizing submodular set
  functions---{I},''
\newblock {\em Mathematical Programming}, vol. 14[1], pp. 265--294, 1978.

\bibitem{Horel16m}
T.~Horel and Y.~Singer,
\newblock ``Maximization of approximately submodular functions,''
\newblock in {\em Advances in Neural Information Processing Systems}, 2016, pp.
  3045--3053.

\bibitem{Lehmann06c}
B.~Lehmann, D.~Lehmann, and N.~Nisan,
\newblock ``Combinatorial auctions with decreasing marginal utilities,''
\newblock {\em Games and Economic Behavior}, vol. 55[2], pp. 270--296, 2006.

\bibitem{Karaca18e}
O.~Karaca and M.~Kamgarpour,
\newblock ``Exploiting weak supermodularity for coalition-proof mechanisms,''
\newblock in {\em Conf. on Decision and Contr.}, 2018, pp. 1118--1123.

\bibitem{Chamon17t}
L.F.O. Chamon, G.J. Pappas, and A.~Ribeiro,
\newblock ``The mean square error in {K}alman filtering sensor selection is
  approximately supermodular,''
\newblock in {\em Conf. on Decision and Contr.}, 2017, pp. 343--350.

\bibitem{Muller99u}
F.~Muller, F.~Seyler, and J.-L. Guyot,
\newblock ``Utilisation d'imagerie radar {(ROS) JERS-1} pour l'obtention de
  réseaux de drainage. {E}xemple du {R}io {N}egro ({A}mazonie),''
\newblock in {\em Hydrological and Geochemical Processes in Large Scale River
  Basins}, 1999,
\newblock
  \url{http://www.ore-hybam.org/index.php/eng/Data/Cartography/Amazon-basin-hydrography}.

\bibitem{Horn13m}
R.A. Horn and C.R. Johnson,
\newblock {\em Matrix analysis},
\newblock Cambridge University Press, 2013.

\end{thebibliography}

\appendix

\section*{Proof of Proposition~\ref{T:LQG}}

\begin{proof}
This result follows directly from the classical dynamic programming argument for the LQG by considering only the inputs in~$\calS$~(see, e.g.,~\cite{Bertsekas17d}). We display the derivations here for ease of reference. Explicitly, we proceed by backward induction, first defining the cost-to-go function
\begin{multline}\label{E:costToGo}
	V_j(\calS) = \min_{\calU_j(\calS)}\ \E \Bigg[
		\sum_{k = j}^{N-1} \left( \bx_k^T \bQ_k \bx_k
			+ \sum_{i \in \calS \cap \calV_k} r_{i,k} u_{i,k}^2 \right)
	\\
	{}+ \bx_N^T \bQ_N \bx_N \Bigg]
		\text{,}
\end{multline}
where~$\calU_j(\calS) = \{u_{i,k} | i \in \calS \cap \calV_k\}_{k=j}^{N-1}$ is the subsequence of control actions from time~$j$ to the end of the control horizon. Notice due to the additivity of~\eqref{E:costToGo}, it admits the equivalent recursive definition
\begin{equation}\label{E:costToGoR}
	V_j(\calS) = \min_{\calU_j(\calS)} \E \Bigg[
		V_{j+1}(\calS) + \bx_{j}^T \bQ_{j} \bx_{j}
		+ \sum_{i \in \calS \cap \calV_{j}} r_{i,{j}} u_{i,{j}}^2
	\Bigg]
		\text{.}
\end{equation}

To proceed, we postulate that~\eqref{E:costToGo} has the form
\begin{equation}\label{E:Vj}
	V_j(\calS) = \bx_j^T \bP_j(\calS) \bx_j + q_j
		\text{,}
\end{equation}
for some~$\bP_j(\calS) \succ 0$ and~$q_j > 0$ and show that this is indeed the case by recursion. To do so, observe that for the base case~$j = N$, \eqref{E:costToGo} becomes~$V_N(\calS) = \bx_N^T \bQ_N \bx_N$ and~\eqref{E:Vj} holds trivially by taking~$\bP_{N}(\calS) = \bQ_N \succ 0$ and~$q_N = 0$. Now, assume that~\eqref{E:Vj} holds for~$j+1$. Using the system dynamics~\eqref{E:dynamics} in~\eqref{E:Vj}, we can expand~$V_{j}(\calS)$ to read
\begin{multline}\label{E:expandedVj}
	V_j(\calS) =
	\min_{\calU_j(\calS)} \E \Bigg[
		\bx_j^T \left(\bQ_j + \bA_j^T \bP_{j+1} \bA_j\right) \bx_j
	\\
	{}+ \left(\sum_{i\in\calS \cap \calV_j} \bb_{i,j} u_{i,j} \right)^T
		\bP_{j+1}(\calS)
	\left( \sum_{i\in\calS \cap \calV_j} \bb_{i,j} u_{i,j} \right)
	\\
	{}+ 2 \bx_j^T \bA_j^T \bP_{j+1}(\calS)
		\left(\sum_{i\in\calS \cap \calV_j} \bb_{i,j} u_{i,j}\right)
	\\
	{}+ 2 \bw_j^T \bP_{j+1}(\calS)
		\left(\sum_{i\in\calS \cap \calV_j} \bb_{i,j} u_{i,j}\right)
	+ 2 \bx_j^T \bA_j^T \bP_{j+1}(\calS) \bw_j
	\\
	{}+ \bw_j^T \bP_{j+1}(\calS) \bw_j
		+ \sum_{i \in \calS \cap \calV_{j}} r_{i,{j}} u_{i,{j}}^2 + q_{j+1}
	\Bigg]
		\text{.}
\end{multline}
Recall that the expected value is taken over the process noises noise~$\bw_k$ and the initial state~$\bx_0$. Note that the terms linear in~$\bw_j$ vanish since it is zero-mean and that~\eqref{E:expandedVj} is actually a quadratic optimization problem in the~$\{u_{i,j}\}$. This is straightforward to see by rewriting~\eqref{E:expandedVj} in matrix form:
\begin{equation}\label{E:expandedVj2}
\begin{aligned}
	V_j(\calS) &= \min_{\bu_j}\
		\bu_{j}^T \left( \bR_{j} + \bB_j^T \bP_{j+1}(\calS) \bB_j \right) \bu_{j}
	\\
	{}&+ 2 \bx_j^T \bA_j^T \bP_{j+1}(\calS) \bB_j \bu_{j}
	\\
	{}&+ \bx_j^T \left(\bQ_j + \bA_j^T \bP_{j+1} \bA_j\right) \bx_j
	\\
	{}&+ \trace\left( \bP_{j+1}(\calS) \bW_j \right) + q_{j+1}
		\text{,}
\end{aligned}
\end{equation}
where~$\bu_j = [u_{i,j}]_{i \in \calS \cap \calV_j}$ is a~$\abs{\calS \cap \calV_j} \times 1$ vector that collects the control actions, $\bB_j = [\bb_{i,j}]_{i \in \calS \cap \calV_j}$ is an~$n \times \abs{\calS \cap \calV_j}$ matrix whose columns contain the input vectors corresponding to each control action, $\bR_j = \diag(r_{i,j})$, and we used the that~$\E \left[ \bw_j^T \bP_{j+1}(\calS) \bw_j \right] = \trace\left( \bP_{j+1}(\calS) \bW_j \right)$. Observe that~\eqref{E:expandedVj2} is classical LQR problem, up to constant terms~\cite{Bertsekas17d}. Its minimum is then of the form~\eqref{E:Vj} with
\begin{equation*}
	\bP_j(\calS) = \bQ_j
	+ \bA_j^T \left( \bP_{j+1}^{-1}(\calS)
	+ \sum_{i \in \calS \cap \calV_j} r_{i,j}^{-1} \bb_{i,j} \bb_{i,j}^T \right)^{-1} \bA_j
		\text{,}
\end{equation*}
and~$q_j = \trace\left( \bP_{j+1}(\calS) \bW_j \right) + q_{j+1}$. Unrolling the recursion and using the fact that~$\E[\bx_0\bx_0^T] = \bSigma_0$ yields the result in~\eqref{E:optimalCost}.
\end{proof}

\section*{Proof of Proposition~\ref{T:alphaBound}}

\begin{proof}
Start by defining
\begin{equation}\label{E:Jk}
	J_k(\calS) = \trace\left[ \bPi_{k-1} \bP_{k}(\calS) \right]
\end{equation}
with~$\bPi_{-1} = \bSigma_0$ and~$\bPi_k = \bW_{k}$ for~$k = 0,\dots,N-2$ and from~\eqref{E:optimalCost}, notice that the objective~$J$ of~\eqref{P:actuatorScheduling} can be written as
\begin{equation}\label{E:expandJk}
	J(\calS) = \sum_{k = 0}^{N-1} J_k(\calS) + \trace\left[ \bW_{N-1} \bQ_{N} \right]
		- V^\star(\emptyset)
		\text{.}
\end{equation}
Hence, using the non-negative affine combination property in Proposition~\ref{T:preserveSM}, we can ignore the terms constant in~$\calS$ and lower bound the approximate supermodularity of~$J$ in terms of the approximate supermodularity of its components. Explicitly, if~$J_k$ in~\eqref{E:Jk} is~$\alpha_k$-supermodular, then~$J$ is~$\min(\alpha_k)$-supermodular.

We can further reduce the problem by using~\eqref{E:P} to get
\begin{equation}\label{E:expandJ}
\begin{aligned}
	J_k(\calS) &= \trace\left[
		\bH_k \left( \bP_{k+1}^{-1}(\calS)
			+ \sum_{i \in \calS \cap \calV_k} r_{i,k}^{-1} \bb_{i,k}^{} \bb_{i,k}^T
		\right)^{-1} \right]
	\\
	{}&+ \trace\left[ \bPi_{k-1} \bQ_k \right]
		\text{.}
\end{aligned}
\end{equation}
where we used the circular shift property of the trace to get~$\bH_k = \bA_{k} \bPi_{k-1} \bA_{k}^T$. Thus, applying Proposition~\ref{T:preserveSM} again, we can ignore the constant term in~\eqref{E:expandJ} and restrict to studying the~$\alpha$-supermodularity of
\begin{equation}\label{E:Jbar1}
	\bar{J}_k(\calS) = \trace\left[
		\bH_k \left( \bP_{k+1}^{-1}(\calS)
			+ \sum_{i \in \calS \cap \calV_k} r_{i,k}^{-1} \bb_{i,k}^{} \bb_{i,k}^T
		\right)^{-1}
	\right]
		\text{.}
\end{equation}

To proceed, notice that~$\bH_k \succ 0$ since~$\bPi_{k} \succ 0$ and~$\bA_k$ is full rank for all~$k$. Hence, it has a unique positive definite square root~$\bH^{1/2} \succ 0$ such that~$\bH = \bH^{1/2}\bH^{1/2}$~\cite{Horn13m}. Deploying the circular shift property of the trace and the invertibility of~$\bH^{1/2}$, \eqref{E:Jbar1} can be written as
\begin{equation}\label{E:Jbar2}
	\bar{J}_k(\calS) = \trace\left[
		\left( \bPt_{k+1}^{-1}(\calS)
			+ \sum_{i \in \calS \cap \calV_k} r_{i,k}^{-1} \bbt_{i,k}^{} \bbt_{i,k}^T
		\right)^{-1}
	\right]
		\text{.}
\end{equation}
with~$\bPt_{k+1}(\calS) = \bH_k^{1/2} \bP_{k+1}(\calS) \bH_k^{1/2}$ and~$\bbt_{i,k} = \bH_k^{-1/2} \bb_{i,k}$. The function~$\bar{J}_k$ in~\eqref{E:Jbar2} has a form that allows us to leverage the following result from~\cite[Thm.~3]{Chamon20a}, which we reproduce here for ease of reference:
\begin{lemma}\label{T:alphaSTF}
Let~$h:2^\calE \rightarrow \setR$ be the set trace function
\begin{equation} \label{E:STF}
	h\left(\calA\right) = \trace\left[\left(\bM_{\emptyset}
		+ \sum_{i \in \calA} \bM_i\right)^{-1}\right]
		\text{,}
\end{equation}
where~$\calA \subseteq \calE$, $\bM_{\emptyset} \succ 0$, and~$\bM_i \succeq 0$ for all~$i \in \calE$. Then, $h$ is (i)~monotonically decreasing and (ii)~$\alpha$-supermodular with
\begin{equation} \label{E:alphaBoundSTF}
	\alpha \geq \frac{\lambda_{\min}\left[\bM_{\emptyset}\right]}{\lambda_{\max}\left[\bM_{\emptyset} + \sum_{i \in \calV} \bM_i\right]} > 0
		\text{.}
\end{equation}
\end{lemma}

\vspace{0.7\baselineskip}

Comparing~\eqref{E:Jbar2} and~\eqref{E:STF}, we can bound the~$\alpha_k$ for~$J_k$ as
\begin{equation}\label{E:preAlpha1}
	\alpha_k \geq \min_{\calS \subseteq \calVbar} \frac{
		\lambda_{\min} \left[ \bPt_{k+1}^{-1}(\calS) \right]
	}{
		\lambda_{\max} \left[ \bPt_{k+1}^{-1}(\calS)
			+ \sum_{i \in \calV_k} r_{i,k}^{-1} \bbt_{i,k}^{} \bbt_{i,k}^T \right]
	}
		\text{.}
\end{equation}
Nevertheless, the bound in~\eqref{E:preAlpha1} depends on the choice of~$\calS$. To obtain a closed-form expression, we use the following proposition:

\begin{proposition}\label{T:explicitBounds}
For any~$\calS \subseteq \calVbar$, it holds that
\begin{equation}\label{E:PBounds}
	\bPt_k(\calVbar) \preceq \bPt_k(\calS) \preceq \bPt_k(\emptyset)
		\text{, for } k = 1,\dots,N-1
		\text{.}
\end{equation}
\end{proposition}

Using~\eqref{E:PBounds} in~\eqref{E:preAlpha1} and the fact that matrix inversion is operator antitone yields the bound in~\eqref{E:alphaBound}.
\end{proof}

All that remains is therefore to prove Proposition~\ref{T:explicitBounds}.

\begin{proof}[Proof of Proposition~\ref{T:explicitBounds}]
Start by noticing that since~$\bH^{1/2}$ is full rank, it is enough to establish~\ref{T:explicitBounds} directly for~$\bP_k$. Indeed, $\bA \preceq \bB \Leftrightarrow \bx^T \bA \bx \leq \bx^T \bB \bx$ for all~$\bx \in \setR^n$ and for~$\bM \succ 0$, this is equivalent to taking~$\bx = \bM \by$ for all~$\by \in \setR^n$.

We prove both inequalities by recursion. For the upper bound in~\eqref{E:PBounds}, note from~\eqref{E:P} that~$\bPt_k$ can be increased by using no actuators at instant~$k$. Formally, for any choice of~$\calS \subseteq \calVbar$,
\begin{align*}
	\bP_k(\calS) &= \bQ_k
			+ \bA_{k}^T \left( \bP_{k+1}^{-1}(\calS)
			+ \sum_{i \in \calS \cap \calV_k} r_{i,k}^{-1} \bb_{i,k}^{} \bb_{i,k}^T \right)^{-1} \bA_{k}
	\\
	{}&\preceq \bQ_k + \bA_k^T \bP_{k+1}(\calS) \bA_k = \bP_k(\calS \setminus \calV_k)
		\text{.}
\end{align*}
Additionally, if~$\bPb_{k+1} \succeq \bP_{k+1}(\calS)$ for all~$\calS \subseteq \calVbar$, it holds that
\begin{equation}\label{E:Pbar}
	\bP_k(\calS) \preceq \bQ_k + \bA_k^T \bPb_{k+1} \bA_k
		\triangleq \bPb_k
		\text{.}
\end{equation}
Starting from~$\bP_N \preceq \bQ_N \triangleq \bPb_N$, we obtain that~$\bP_k$ is upper bounded by taking~$\calS = \emptyset$, i.e., using no actuators.

The lower bound is obtained in a similar fashion by using all possible actuators. Explicitly,
\begin{align*}
	\bP_k(\calS) &= \bQ_k
			+ \bA_{k}^T \left( \bP_{k+1}^{-1}(\calS)
			+ \sum_{i \in \calS \cap \calV_k} r_{i,k}^{-1} \bb_{i,k}^{} \bb_{i,k}^T \right)^{-1} \bA_{k}
	\\
	{}&\succeq \bQ_k
			+ \bA_{k}^T \left( \bP_{k+1}^{-1}(\calS)
			+ \sum_{i \in \calV_k} r_{i,k}^{-1} \bb_{i,k}^{} \bb_{i,k}^T \right)^{-1} \bA_{k}
	\\
	{}&= \bP_k(\calS \cup \calV_k)
		\text{.}
\end{align*}
Moreover, if~$\bPu_{k+1} \preceq \bP_{k+1}(\calS)$ for all~$\calS \subseteq \calVbar$, we have that
\begin{equation}
	\bP_k(\calS) \succeq \bQ_k
			+ \bA_{k}^T \left( \bPu_{k+1}^{-1}
			+ \sum_{i \in \calV_k} r_{i,k}^{-1} \bb_{i,k}^{} \bb_{i,k}^T \right)^{-1} \bA_{k}
		\triangleq \bPu_{k}
		\text{.}
\end{equation}
Starting from~$\bP_N \succeq \bQ_N \triangleq \bPu_N$ yields the lower bound in~\eqref{E:PBounds}.
\end{proof}

% ========== BIOGRAPHIES ==========
\begin{IEEEbiography}[{\includegraphics[width=1in,height=1.25in]{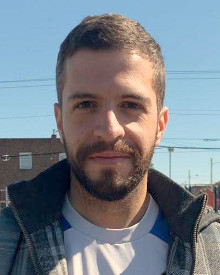}}]{Luiz F. O. Chamon (S'12--M'21)}
received the B.Sc.\ and M.Sc.\ degrees in electrical engineering from the University of S\~{a}o Paulo Paulo, S\~{a}o Paulo Paulo, Brazil, in 2011 and 2015 and the Ph.D.\ degree in electrical and systems engineering from the University of Pennsylvania (Penn), Philadelphia, in 2020. He is currently a postdoctoral researcher of the Department of Electrical and Systems Engineering of the University of Pennsylvania. In 2009, he was an undergraduate exchange student of the Masters in Acoustics of the \'{E}cole Centrale de Lyon, Lyon, France, and worked as an Assistant Instructor and Consultant on nondestructive testing at INSACAST Formation Continue. From 2010 to 2014, he worked as a Signal Processing and Statistics Consultant on a research project with EMBRAER. In 2018, he was recognized by the IEEE Signal Processing Society for his distinguished work for the editorial board of the IEEE Transactions on Signal Processing. He also received both the best student paper and the best paper awards at IEEE ICASSP 2020. His research interests include optimization, signal processing, machine learning, statistics, and control.
\end{IEEEbiography}

\begin{IEEEbiography}[{\includegraphics[width=1in,height=1.25in]{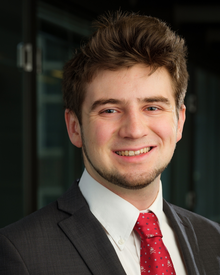}}]{Alexandre Amice}
received the B.S.E.\ in Electrical Engineering and Mathematics and the M.S.E.\ in Robotics from the University of Pennsylvania, Pennsylvania, USA in 2020. His research interests include optimization, dynamical systems, and control.
\end{IEEEbiography}

\begin{IEEEbiography}[{\includegraphics[width=1in,height=1.25in]{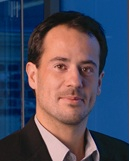}}]{Alejandro Ribeiro}
received the B.Sc.\ degree in electrical engineering from the Universidad de la Republica Oriental del Uruguay, Montevideo, in 1998 and the M.Sc.\ and Ph.D.\ degree in electrical engineering from the Department of Electrical and Computer Engineering, the University of Minnesota, Minneapolis in 2005 and 2007.

From 1998 to 2003, he was a member of the technical staff at Bellsouth Montevideo. After his M.Sc. and Ph.D studies, in 2008 he joined the University of Pennsylvania (Penn), Philadelphia, where he is currently the Rosenbluth Associate Professor at the Department of Electrical and Systems Engineering. His research interests are in the applications of statistical signal processing to the study of networks and networked phenomena. His focus is on structured representations of networked data structures, graph signal processing, network optimization, robot teams, and networked control.

Dr. Ribeiro received the 2014 O. Hugo Schuck best paper award, and paper awards at CDC 2017, 2016 SSP Workshop, 2016 SAM Workshop, 2015 Asilomar SSC Conference, ACC 2013, ICASSP 2006, and ICASSP 2005. His teaching has been recognized with the 2017 Lindback award for distinguished teaching and the 2012 S. Reid Warren, Jr. Award presented by Penn's undergraduate student body for outstanding teaching. Dr. Ribeiro is a Fulbright scholar class of 2003 and a Penn Fellow class of 2015.
\end{IEEEbiography}

\end{document}